\documentclass[12pt]{article}
\usepackage{amsmath,amsthm, amssymb,
extsizes}
\usepackage{amscd,array}
\usepackage[all]{xy}
\usepackage[active]{srcltx}

\title{Canonical matrices of
forms and pairs of forms over finite
and $\mathfrak p$-adic fields}
\author{
Vladimir V. Sergeichuk
\\
Institute of
Mathematics,\\
Tereshchenkivska 3,
Kiev,
Ukraine,\\sergeich@imath.kiev.ua}
\date{}

\begin{document}

\renewcommand{\le}{\leqslant}
\renewcommand{\ge}{\geqslant}
\newcommand{\diag}{\,\diagdown\,}

\newcommand{\lin}{\,\frac{}{\quad}\,}

\newcommand{\is}{\stackrel
{\text{\raisebox{-1ex}{$\sim\
\;$}}}{\to}}

\newcommand{\ind}%
{\mathop{\rm
ind}\nolimits}

\newcommand{\ci}{
\begin{picture}(6,6)
\put(3,3){\circle*{3}}
\end{picture}}

\newcommand{\ddd}{
\text{\begin{picture}(12,8)
\put(-2,-4){$\cdot$}
\put(3,0){$\cdot$}
\put(8,4){$\cdot$}
\end{picture}}}

\newcommand{\sdotss}%
{\text{\raisebox{-1.8pt}{$\cdot\,$}%
 \raisebox{1.5pt}{$\cdot$}%
\raisebox{4.8pt}{$\,\cdot$}}}

\newtheorem{theorem}{Theorem}[section]
\newtheorem{lemma}[theorem]{Lemma}
\newtheorem{corollary}[theorem]{Corollary}

\theoremstyle{remark}
\newtheorem{remark}[theorem]{Remark}

\maketitle
\begin{abstract}
Canonical matrices of
\begin{itemize}
  \item[(a)]
bilinear and
sesquilinear forms,
\item[(b)] pairs of
forms, in which every
form is symmetric or
skew-symmetric, and
\item[(c)] pairs of
Hermitian forms
\end{itemize}
are given over finite fields of
characteristic $\ne 2$ and over
$\mathfrak p$-adic fields (i.e., finite
extensions of the field $\mathbb Q_p$
of $p$-adic numbers) with $p\ne 2$.

These canonical matrices are special
cases of the canonical matrices of
(a)--(c) over a field of characteristic
not $2$ that were obtained by the
author [{\it Math. USSR--Izv.} 31
(1988) 481--501] up to classification
of quadratic or Hermitian forms over
its finite extensions; we use the known
classification of quadratic and
Hermitian forms over finite fields and
$\mathfrak p$-adic fields.


 \end{abstract}

\section{Introduction}
\label{intr}

We give canonical
matrices of
\begin{itemize}
  \item[(a)]
bilinear and
sesquilinear forms,

  \item[(b)]
pairs of forms in
which every form is
symmetric or
skew-symmetric, and

  \item[(c)]
pairs of Hermitian
forms
\end{itemize}
over
\begin{itemize}
  \item[(i)]
finite fields of
characteristic
different from $2$,
and
  \item[(ii)]
$\mathfrak p$-adic fields (i.e., finite
extensions of the field $\mathbb Q_p$
of $p$-adic numbers); for simplicity,
we take $p\ne 2$.
\end{itemize}

Our canonical matrices are special
cases of the canonical matrices of
(a)--(c) over a field $\mathbb F$ of
characteristic not $2$ that were
obtained in \cite{ser_izv} up to
classification of quadratic or
Hermitian forms over finite extensions
of $\mathbb F$. We use the known
classification of quadratic and
Hermitian forms over finite extensions
of (i) and (ii).

 Analogous canonical
matrices of (a)--(c) could be obtained
over any local field (which is either a
$\mathfrak p$-adic field or the field
of formal power series of one variable
over a finite field) since the
classification of quadratic and
Hermitian forms over local fields is
known.

In Section
\ref{any} we recall
canonical forms of
(a)--(c) obtained in
\cite{ser_izv}.
In Sections
\ref{fin} and
\ref{loc} we give
canonical forms of
(a)--(c) over (i) and
(ii).

\section{Canonical
matrices over any
field of
characteristic not
2}\label{any}

In this section
$\mathbb F$ denotes a
field of
characteristic
different from 2 with
a fixed involution
$\mathbb F\to \mathbb
F$; that is, a
bijection $a\mapsto
\bar{a}$ satisfying
\[\overline{a+b}=\bar
a+ \bar b,\quad
\overline{ab}=\bar a
\bar b,\quad\bar{\bar
a}=a\qquad\text{for
all $a,b\in\mathbb
F$.}\]
We recall
canonical forms of
(a)--(c) obtained in
\cite{ser_izv} by
the method that was
developed by Roiter
and the author in
\cite{roi,ser_first,ser_izv};
it reduces the problem
of classifying systems
of forms and linear
mappings over $\mathbb
F$ to the problems of
classifying
\begin{itemize}
  \item
systems of linear
mappings over $\mathbb F$, and
  \item
quadratic and Hermitian
forms over skew fields that are finite
extensions of $\mathbb
F$.
\end{itemize}
This method was
applied to the problem
of classifying
bilinear and
sesquilinear forms in
\cite{hor-ser,hor-ser_can,h-s_bilin_anyf}
and to the problem of
classifying isometric
operators on vector
spaces with scalar
product given by a
nonsingular quadratic
or Hermitian form in
\cite{ser_iso}.

For any matrix
$A=[a_{ij}]$ over
$\mathbb F$, we write
$
A^*:=\bar{A}^T=[\bar{a}_{ji}].$
Square matrices
$A$ and $B$ are
said to be
\emph{similar} if
$S^{-1}AS=B$,
\emph{congruent} if
$S^TAS=B,$ and
\emph{*congruent} if
$S^*AS=B$ for a
nonsingular
$S$. Pairs
of matrices
$(A_1,A_2)$ and
$(B_1,B_2)$ are
\emph{congruent} if
$S^TA_1S=B_1$ and
$S^TA_2S=B_2$; they
are \emph{*congruent}
if $S^*A_1S=B_1$ and
$S^*A_2S=B_2$ for a
nonsingular
$S$. The
transformations of
congruence ($A\mapsto
S^TAS$) and
*congruence ($A\mapsto
S^*AS$) are associated
with the bilinear form
$x^TAy$ and the
sesquilinear form
$x^*Ay$, respectively.

The involution on
$\mathbb F$ can be the
identity. Thus, we
consider congruence as
a special case of
*congruence.

Every square matrix
$A$ over $\mathbb F$
is similar to a direct
sum, uniquely
determined up to
permutation of
summands, of {\it
Frobenius blocks}
\begin{equation}\label{3}
{\Phi}=\begin{bmatrix}
0&&
0&-c_m\\1&\ddots&&\vdots
\\&\ddots&0&-c_2\\
0&&1& -c_1
\end{bmatrix},
\end{equation}
whose characteristic
polynomial
\[\chi_{\Phi}(x)=p_{\Phi}(x)^l=x^m+c_1
x^{m-1}+\dots+ c_m\]
is an integer power of
a polynomial
$p_{\Phi}(x)$ that is
irreducible over
$\mathbb F$; this
direct sum is called
the \emph{Frobenius
canonical form} or the
\emph{rational
canonical form} of
$A$, see
\cite[Section 6]{b-r}.
If
$\chi_{\Phi}(x)=(x-\lambda
)^m$, then ${\Phi}$ is
similar to the Jordan
block
\begin{equation}\label{fik}
J_m(\lambda
):=\begin{bmatrix}
\lambda&&&
0\\1&\lambda
\\&\ddots&\ddots\\
0&&1&\lambda
\end{bmatrix}\qquad
(m\text{-by-}m).
\end{equation}

For each polynomial
\[
f(x)=a_0x^n+a_1x^{n-1}+\dots
+a_n\in \mathbb F[x],
\]
we define the
polynomials
\begin{align}\label{iut}
\bar f(x)&:=\bar
a_0x^n+\bar
a_1x^{n-1}+\dots+\bar
a_n,\\\label{ksu}
f^{\vee}(x)&:=\bar
a_n^{-1}(\bar
a_nx^n+\dots+\bar
a_1x+\bar
a_0)\quad\text{if }
a_n\ne 0.
\end{align}
In particular,
\begin{equation}\label{ksu1}
f^{\vee}(x)=
a_n^{-1}(
a_nx^n+\dots+
a_1x+
a_0)
\end{equation}
if the involution on $\mathbb F$ is the identity.

The following lemma
was proved in
\cite[Lemma
6]{ser_izv} (or see
\cite{h-s_bilin_anyf,ser_iso}).

\begin{lemma}
\label{LEMMA 7} Let
$\mathbb F$ be a field
with involution
$a\mapsto \bar a$, let
$p(x) = p^{\vee}(x)$
be an irreducible
polynomial over
$\mathbb F$, and consider the
field
\begin{equation}\label{alft}
\mathbb F(\kappa) =
\mathbb
F[x]/p(x)\mathbb
F[x],\qquad \kappa:=
x+p(x)\mathbb F[x],
\end{equation}
with involution
\begin{equation}\label{alfta}
f(\kappa)^{\circ} :=
\bar f(\kappa^{-1}).
\end{equation}
Then each element of\/
$\mathbb F(\kappa)$ on
which the involution
acts identically is
uniquely representable
in the form
$q(\kappa)$, in which
\begin{equation}\label{ser13}
q(x)=a_rx^r+\dots+
a_1x +a_0+\bar
a_1x^{-1}+\dots+\bar
a_rx^{-r},
    \quad a_0 = \bar a_0,
\end{equation}
$r$ is the integer
part of $(\deg
p(x))/2$, $a_0,\dots,
a_r\in\mathbb F,$ and
if $\deg p(x)$ is even
then
\begin{equation*}\label{uvp}
a_r=
  \begin{cases}
    0
&\text{if the
involution
on $\mathbb F$ is the identity}, \\
    \bar a_r
&\text{if the
involution on $\mathbb
F$ is not the identity
and
$p(0)\ne 1$},\\
    -\bar a_r
&\text{if the
involution on $\mathbb
F$ is not the identity
and $p(0)=1$}.
  \end{cases}
\end{equation*}
\end{lemma}

For each square matrix $\Phi$
and
\[
\varepsilon=
  \begin{cases}
    1\text{ or }-1, & \text{if
the involution on
$\mathbb F$ is the
identity}, \\
    1, & \text{if
the involution on
$\mathbb F$ is
nonidentity},
  \end{cases}
\]
denote
by
$\sqrt[\displaystyle
*]{\Phi}$ and
$\Phi_{\varepsilon}$
fixed nonsingular
matrices (if they exist) such that
\begin{gather}\label{vr1}
\sqrt[\displaystyle
*]{\Phi}=(\sqrt[\displaystyle
*]{\Phi})^*\Phi,
   \\ \label{vr2}
\Phi_{\varepsilon}
=\Phi_{\varepsilon}^*,\qquad
\Phi_{\varepsilon}\Phi=
\varepsilon
(\Phi_{\varepsilon}\Phi)^*.
\end{gather}
We use the notation $\sqrt[\displaystyle
*]{\Phi}$ both in the case of nonidentity involution and in the case of the identity involution on $\mathbb F$, but if we know that the involution is the identity then we prefer  to write
$\sqrt[T]{\Phi}$
instead of
$\sqrt[\displaystyle
*]{\Phi}$.

It suffices
to construct
$\sqrt[\displaystyle
*]{\Phi}$ and
$\Phi_{\varepsilon}$
for canonical matrices
$\Phi$ under similarity since if $\Psi
=S^{-1}\Phi S$ then we
can take
\begin{equation*}\label{ndw}
\sqrt[\displaystyle
*]{\Psi}
=S^*\sqrt[\displaystyle
*]{\Phi}S,\qquad
\Psi_{\varepsilon}
=S^*\Phi_{\varepsilon}S.
\end{equation*}
Existence conditions
and explicit forms of
$\sqrt[\displaystyle
*]{\Phi}$ and
$\Phi_{\varepsilon}$
for all Frobenius
blocks $\Phi$ will be
given in Lemmas
\ref{lsdy1} and
\ref{THEjM}.

Define the \textit{skew sum} of two matrices
\[
\lbrack A\,\diagdown\,B]:=%
\begin{bmatrix}
0 & B\\
A & 0
\end{bmatrix}
.
\]

\begin{theorem}[{\cite[Theorem
3]{ser_izv}; see also
\cite[Theorem
2.2]{h-s_bilin_anyf}}]
\label{Theorem 5}

\begin{itemize}
  \item[{\rm(a)}] Let\/
$\mathbb F$ be a field
of characteristic
different from $2$
with involution
$($which can be the
identity$)$. Every
square matrix $A$
over\/ $\mathbb F$ is
*congruent to a direct
sum of matrices of the
following types:
\begin{itemize}
 \item [{\rm(i)}]
$J_n(0)$;

 \item [{\rm(ii)}]
$[\Phi\diag I_n]$,
where $\Phi$ is an
$n\times n$
nonsingular Frobenius
block such that
$\sqrt[\displaystyle
*]{\Phi}$ does not
exist $($see Lemma
{\rm\ref{lsdy1}}$)$;
and

 \item [{\rm(iii)}]
$\sqrt[\displaystyle
*]{\Phi}q(\Phi)$,
where $\Phi$ is a
nonsingular Frobenius
block such that
$\sqrt[\displaystyle
*]{\Phi}$ exists and
$q(x)\ne 0$ has the
form \eqref{ser13}
from Lemma \ref{LEMMA
7} in which
$p(x)=p_{\Phi}(x)$ is
the irreducible
divisor of the
characteristic
polynomial of $\Phi$.
\end{itemize}

  \item[{\rm(b)}] The summands
are determined to the
following extent:
\begin{description}
  \item [Type (i)]
uniquely.

  \item [Type (ii)]
up to replacement of
$\Phi$ by the
Frobenius block
$\Psi$ that is similar
to $\Phi^{-*}$
$($i.e., whose
characteristic
polynomial is
$\chi_{\Phi}^{\vee}(x)$, see \eqref{ksu}$)$.

  \item [Type (iii)]
up to replacement of
the whole group of
summands
\[
\sqrt[\displaystyle
*]{\Phi}q_1(\Phi)
\oplus\dots\oplus
\sqrt[\displaystyle
*]{\Phi}q_s(\Phi)
\]
with the same $\Phi$
by
\[
\sqrt[\displaystyle
*]{\Phi}q'_1(\Phi)
\oplus\dots\oplus
\sqrt[\displaystyle
*]{\Phi}q'_s(\Phi)
\]
in which each
$q'_i(x)$ is a nonzero
function of the form
\eqref{ser13} and the
Hermitian forms
\begin{gather*}\label{777}
q_1(\kappa)x_1^{\circ}y_1+\dots+
q_s(\kappa)x_s^{\circ}y_s,
\\\label{777s}
q'_1(\kappa)x_1^{\circ}y_1+\dots+
q'_s(\kappa)x_s^{\circ}y_s
\end{gather*}
are equivalent over
the field \eqref{alft}
with involution
\eqref{alfta}.
\end{description}

  \item[{\rm(c)}]
Frobenius blocks in {\rm(a)} and
{\rm(b)} can be replaced by arbitrary matrices that are similar to them $($for example, by Jordan blocks if\/ $\mathbb F$ is algebraically closed$)$.
\end{itemize}
\end{theorem}

Define
the $(n-1)\times n$
matrices
\begin{equation}\label{ser17}
F_n:=\begin{bmatrix}
1&0&&0\\&\ddots&\ddots&\\0&&1&0
\end{bmatrix},\quad
G_n:=\begin{bmatrix}
0&1&&0\\&\ddots&\ddots&\\0&&0&1
\end{bmatrix}
\end{equation}
for each
$n=1,2,\dots$, and
define the \textit{direct sum} of two matrix pairs:
\[
(A_{1},B_{1})\oplus(A_{2},B_{2}):=(A_{1}\oplus A_{2},\,B_{1}\oplus
B_{2}).
\]

\begin{theorem}[{\cite[Theorem
4]{ser_izv}}]
\label{The4}
\begin{itemize}
  \item[{\rm(a)}]
Let\/ $\mathbb F$ be a
field of
characteristic
different from $2$
with involution $($which can be the identity$)$. Let
$A$ and $B$ be
$\varepsilon$-Hermitian
and $\delta$-Hermitian
matrices over $\mathbb
F$ of the same size:
\begin{equation*}\label{lut}
A^*=\varepsilon
A,\qquad B^*=\delta B,
\end{equation*}
in which
\[
(\varepsilon,\delta)=
  \left\{\begin{array}{l}
  \!\!(1,1), \text{ if the
involution on $\mathbb
F$ is nonidentity},
\\
  \!\!   (1,1)\text{ or }(1,-1)\text{ or }(-1,-1) , \text{ otherwise}.
\end{array}
\right.
\]
Then $(A,B)$ is
*congruent to a direct
sum of matrix pairs of
the following types:
\begin{itemize}
  \item[\rm(i)]
$([F_n\diag
\varepsilon F_n^*],\,
[G_n\diag \delta
G_n^*])$, in which
$F_n$ and $G_n$ are
defined in
\eqref{ser17};

  \item[\rm(ii)]
$([I_n\diag
\varepsilon I_n],\,
[\Phi\diag \delta
\Phi^*])$, in which
$\Phi$ is an $n\times
n$ Frobenius block
such that
$\Phi_{\delta}$ $($see
\eqref{vr2}$)$ does
not exist if
$\varepsilon=1$;

  \item[\rm(iii)]
$A_{\Phi}^{f(x)}:=
(\Phi_{\delta},
\Phi_{\delta}\Phi)f(\Phi)$ only if $\varepsilon=1$,
in which
$0\ne
f(x)=\bar f(\delta
x)\in \mathbb F[x]$ $($see \eqref{iut}$)$,
and
$\deg(f(x))<\deg(p_{\Phi}(x))$;

  \item[\rm(iv)]
$([J_n(0)\diag
\varepsilon J_n(0)^*],
[I_n\diag (-I_n)])$ only if
$\delta=-1$,
in which $n$ is odd if
$\varepsilon=1$;

  \item[\rm(v)]
\begin{equation}\label{ser120}
\arraycolsep=4.4pt
B_n^a:=\left(\!
 a\!\begin{bmatrix}
0&&&&&1&0\\
&&&&\delta&\cdot&\\
&&&1&\cdot&&\\
&&\delta&\cdot&&&\\
&\cdot&\cdot&&&&\\
\cdot&\cdot&&&&&\\
0&&&&&&0
\end{bmatrix},\
a\!\begin{bmatrix}
0&&&&&&1\\
&&&&&\delta&\\
&&&&1&&\\
&&&\delta&&&\\
&&\cdot&&&&\\
&\cdot&&&&&\\
\cdot&&&&&&0
\end{bmatrix}
 \!\right),
\end{equation}
in which the matrices
are $n$-by-$n$,
$\varepsilon=1$, $0\ne
a=\bar a\in \mathbb
F$, and $n$ is even if
$\delta=-1$.
\end{itemize}

  \item[{\rm(b)}] The summands
are determined to the
following extent:
\begin{description}
  \item [Type (i)]
uniquely.

  \item [Type (ii)]
up to replacement of
$\Phi$ by the Frobenius block $\Psi$ with
$\chi_{\Psi}(x)
=(\varepsilon
\delta)^{\det \chi_{\Phi}} \bar\chi_{\Phi}(\varepsilon
\delta x)$.

  \item [Type (iii)]
up to replacement of
the whole group of
summands
\[
A_{\Phi}^{f_1(x)}
\oplus\dots\oplus
A_{\Phi}^{f_s(x)}
\]
with the same $\Phi$
by
\[
A_{\Phi}^{g_1(x)}
\oplus\dots\oplus
  A_{\Phi}^{g_s(x)}
\]
such that the
Hermitian forms
\begin{gather*}\label{777aki}
f_1(\omega)x_1^{\circ}y_1+\dots+
f_s(\omega)x_s^{\circ}y_s,
\\\label{777sopja}
g_1(\omega)x_1^{\circ}y_1+\dots+
g_s(\omega)x_s^{\circ}y_s
\end{gather*}
are equivalent over
the field\/ ${ \mathbb
F}(\omega)={ \mathbb
F}[x]/p_{\Phi}(x){\mathbb
F}[x]$ with involution
$f(\omega)^{\circ}=
\bar f(\delta\omega)$.

  \item [Type (iv)]
uniquely.

  \item [Type (v)]
up to replacement of
the whole group of
summands
\begin{equation}\label{tek}
B_n^{a_1}
\oplus\dots\oplus
B_n^{a_s}
\end{equation}
with the same $n$ by
\begin{equation}\label{ftl}
B_n^{b_1}
\oplus\dots\oplus
B_n^{b_s}
\end{equation}
such that the
Hermitian forms
\begin{gather*}\label{777a}
a_1\bar x_1y_1+\dots+
a_s\bar x_sy_s,
\\\label{777sa}
b_1\bar x_1y_1+\dots+
b_s\bar x_sy_s
\end{gather*}
are equivalent over\/
${\mathbb F}$.
\end{description}

  \item[{\rm(c)}] Frobenius blocks in {\rm(a)} and
{\rm(b)} can be replaced by arbitrary matrices that are similar to them $($for example, by Jordan blocks if\/ $\mathbb F$ is algebraically closed$)$.
\end{itemize}
\end{theorem}

Taking $\varepsilon =
\delta=-1$ in Theorem
\ref{The4}, we obtain
the following
well-known canonical
form of pairs
skew-symmetric
matrices; see, for
example, \cite{r.sch,tho}.

\begin{corollary}
\label{Theorem 4} Over
any field of
characteristic not
$2$, each pair of
skew-symmetric
matrices of the same
size is congruent to a
direct sum, uniquely
determined up to
permutation of
summands, of pairs of
the form:
\begin{itemize}
  \item[\rm(i)]
$([F_n\diag -
F_n^T],\, [G_n\diag -
G_n^T])$, in which $F_n$
and $G_n$ are defined
in \eqref{ser17};

  \item[\rm(ii)]
$([I_n\diag - I_n],\,
[\Phi\diag -
\Phi^T])$,  in which $\Phi$ is an $n\times
n$ Frobenius block;

  \item[\rm(iii)]
$([J_n(0)\diag -
J_n(0)^T], [I_n\diag
-I_n])$.
\end{itemize}
\end{corollary}

\begin{remark}\label{hts}
If $\delta =-1$ then
the matrix pair
$B_n^a$ defined in
\eqref{ser120}
consists of $n\times
n$ matrices and $n$ is
even. In this case,
the pair
\begin{equation}\label{serk1}
C_n^a:=\left(\!
a\! \begin{bmatrix}
0&&1&0\\
&\ddd&\ddd\\
1&0\\
0&&&0
\end{bmatrix},\
a\!\begin{bmatrix}
0&&&&&1\\
&&&&\ddd&\\
&&&1&&\\
&&-1&&&\\
&\ddd&&&&\\
-1&&&&&0\\
\end{bmatrix}
 \!\right)
\end{equation}
of symmetric and
skew-symmetric
matrices of size
$n\times n$ can be
used in
\eqref{ser120}--\eqref{ftl} instead of
$B_n^a$. This follows
from the proof of
Theorem 4 in
\cite{ser_izv} since
the pairs $B_n^a$ and
$C_n^a$ are
equivalent; that is,
$RB_n^aS=C_n^a$ for
some nonsingular $R$
and $S$.
\end{remark}

Let \[f(x)=
\gamma_0x^m +
\gamma_1x^{m-1}+\dots+\gamma_m\in
\mathbb F[x],\qquad m\ge 1,\
\gamma_0\ne
0\ne\gamma_m.\] A
vector $(a_1,
a_{2},\dots, a_n)$
over $\mathbb F$ is
called
\emph{$f$-recurrent}
if either $n\le m$, or
\[
\gamma_0 a_{l} +
\gamma_{1}a_{l+1}+\dots+
\gamma_ma_{l+m}=0\qquad\text{for all }
l=1,2,\dots,n - m.
\]
Thus, this
vector is completely
determined by any
fragment of length
$m$.

Existence conditions
and explicit forms of
$\sqrt[\displaystyle
*]{\Phi}$ and
$\Phi_{\varepsilon}$ for Frobenius blocks $\Phi$
are given in the
following two lemmas.

\begin{lemma}[{\cite[Theorem
7]{ser_izv}; a
detailed proof in
\cite[Lemma
2.3]{h-s_bilin_anyf}}]
\label{lsdy1} Let
$\mathbb F$ be a field
of characteristic not
$2$ with involution
$($possibly, the
identity$)$. Let
$\Phi$ be an ${n\times
n}$ nonsingular
Frobenius block whose
characteristic
polynomial is a power
of an irreducible
polynomial
$p_{\Phi}(x)$.

\begin{itemize}
  \item[{\rm(a)}]
$\sqrt[\displaystyle
*]{\Phi}$ exists if
and only if
\begin{equation}\label{lbdr}
 p_{\Phi}(x) =
p_{\Phi}^{\vee}(x)\ (\text{see \eqref{ksu}}),\ \
\text{and}
\end{equation}
\begin{equation}\label{4.adlw}
\text{if the
involution on $\mathbb
F$ is the identity,
also $p_{\Phi}(x)\ne x
+ (-1)^{n+1}$}.
\end{equation}

  \item[{\rm(b)}] If
\eqref{lbdr} and
\eqref{4.adlw} are
satisfied and
\begin{equation}\label{ser24lk}
\chi_{\Phi}(x)=x^n+
c_1x^{n-1}+\dots+c_n
\end{equation}
is the
characteristic
polynomial of
${\Phi}$, then for
$\sqrt[\displaystyle
*]{\Phi}$ one can take
the Toeplitz matrix
\begin{equation}\label{okjd}
\sqrt[\displaystyle
*]{\Phi}:= [a_{i-j}]=
\begin{bmatrix}
a_0
&a_{-1}&\ddots&a_{1-n}
\\a_{1}&a_0
&\ddots&\ddots
\\\ddots&\ddots&
\ddots&a_{-1}
\\ a_{n-1}&\ddots&
a_{1}&a_0
\end{bmatrix},
\end{equation}
whose vector of
entries
$(a_{1-n},a_{2-n},\dots,a_{n-1})$
is the
$\chi_{\Phi}$-recurrent
extension of the
vector
\begin{equation}\label{ksy}
v=(a_{1-m},\dots,a_{m})
=(a,0,\dots,0,\bar a)
\end{equation}
of length
\begin{equation}\label{leg}
2m=
  \begin{cases}
    n & \text{if $n$ is even}, \\
    n+1 & \text{if $n$ is
odd,}
  \end{cases}
\end{equation}
in which
\begin{equation}
 \label{mag}
a:=
  \begin{cases}
1 & \text{if $n$ is
even,
    except for the
    case}
    \\
    &
\qquad
p_{\Phi}(x)=x+c\
\text{with }c
^{n-1}=-1,
          \\
    \chi_{\Phi}(-1)&
\text{if $n$ is odd
and $p_{\Phi}(x)\ne
x+1$,}
           \\
e-\bar
e&\text{otherwise,
with any fixed $\bar
e\ne e\in\mathbb F$}.
  \end{cases}
\end{equation}
\end{itemize}
\end{lemma}

\begin{lemma}[{\cite[Theorem
8]{ser_izv}}]\label{THEjM}
 Let
$\mathbb F$ be a field
of characteristic not
$2$ with involution
$($possibly, the
identity$)$. Let
$\Phi$ be an $n\times
n$ Frobenius block
\eqref{3} over
$\mathbb F$.
 Existence
conditions for the
matrix
${\Phi}_{\varepsilon}$
are:
\begin{gather}\label{lit}
p_{\Phi}(x) =
\varepsilon^n\bar
p_{\Phi} (\varepsilon
x)\ (\text{see \eqref{iut}}),
  \\
\text{if }
\varepsilon=-1\
\text{then also
}\chi_{\Phi}(x)\notin\{x^2,x^4,x^6,\ldots\}.
\end{gather}

With these conditions
satisfied, one can
take
\[
{\Phi}_{\varepsilon}=
[\varepsilon^i
a_{i+j}],
\]
in which the sequence
$(a_2,a_3,\dots,a_{2n})$
is $\chi$-recurrent,
and is defined by the
fragment
\begin{equation}\label{msu1}
(a_2,\dots,a_{n+1})=
  \begin{cases}
    (1,0,\dots,0) & \text{if $\Phi$
    is nonsingular}, \\
    (0,\dots,0,1) & \text{if $\Phi$
    is singular}.
  \end{cases}
\end{equation}
\end{lemma}

\section{Canonical
forms over finite
fields}\label{fin}

In this section we
give canonical
matrices of bilinear
and sesquilinear
forms, pairs of
symmetric or
skew-symmetric forms,
and pairs of Hermitian
forms over a finite
field $\mathbb F$ of
characteristic not
$2$. We use Theorems
\ref{Theorem 5} and
\ref{The4}, in which
these canonical
matrices are given up
to classification of
quadratic and
Hermitian forms over
finite extensions of
$\mathbb F$ (that is,
over finite fields of
characteristic not
$2$), and the
following lemma.

\begin{lemma}
[{\cite[Chap.\,1,
\S\,8]{die}}]\label{THt}

\begin{itemize}
  \item[{\rm(a)}] Each
quadratic form of rank $r$ over a
finite field  $\mathbb F$ of
characteristic not $2$ is equivalent to
\[\text{either }\
x_1^2+ x_2^2+\dots
+x_r^2,\qquad\text{or }\ \zeta x_1^2+
x_2^2+\dots +x_r^2,
\]
where $\zeta $ is a fixed nonsquare in
$\mathbb F$.

  \item[{\rm(b)}] Each
Hermitian form of rank $r$ over a
finite field of characteristic not $2$
with nonidentity involution is
equivalent to $ \bar x_1 y_1+\dots
+\bar x_r y_r. $
\end{itemize}
\end{lemma}

Utv. (b) eshe iz Scharlau ch 10, 1.6,
examples (i).

\subsection{Canonical matrices for
congruence and
*congruence}

Define the $n$-by-$n$
matrix
\begin{equation}\label{kus}
\Gamma_{n}=%
\begin{bmatrix}
0&  &  &  & \ddd
\\
&  &  & -1 &
\ddd \\
&  & 1 & 1 \\
& -1 & -1\\
1 & 1  &  &  & 0
\end{bmatrix}
\qquad
(\Gamma_{1}=[\,1\,]).
\end{equation}%

\begin{theorem}
\label{Th5} Every square matrix over a
finite field\/ $\mathbb F$ of
characteristic different from $2$ is
congruent to a direct sum that is
uniquely determined up to permutation
of summands and consists of any number
of summands of the following types:
\begin{itemize}
 \item [{\rm(i)}]
$J_n(0)$;

 \item [{\rm(ii)}]
$[\Phi\diag I_n]$, in
which $\Phi$ is an
$n\times n$
nonsingular Frobenius
block such that
\begin{equation}\label{fus}
p_{\Phi}(x) \ne
p_{\Phi}^{\vee}(x)\ (\text{see \eqref{ksu1}})\qquad
\text{or}\qquad
p_{\Phi}(x)= x +
(-1)^{n+1},
\end{equation}
and $\Phi$ is
determined up to
replacement by the
Frobenius block $\Psi$
with $\chi_{\Psi}(x)
=\chi^{\vee}_{\Phi}(x)$;

 \item [{\rm(iii)}]
$\sqrt[T]{\Phi}$, in which $\Phi$ is a nonsingular
Frobenius block
such that $p_{\Phi}(x)
= p_{\Phi}^{\vee}(x)$
and $\deg
p_{\Phi}(x)\ge 2$;

 \item [{\rm(iv)}]
for each
$n=1,2,\dots$:
\begin{itemize}
  \item[$\bullet$]
$\Gamma_{n}$,
\item[$\bullet$] at
most one summand
$\zeta \Gamma_{n}$, in
which $\zeta$ is a
fixed nonsquare of
$\mathbb F$.
\end{itemize}
\end{itemize}
\end{theorem}

\begin{proof}
Let $\mathbb F$ be a
finite field of
characteristic not $2$
with the identity
involution.
By Theorem
\ref{Theorem 5}(a), every
square matrix $A$ over
$\mathbb F$ is
congruent to a direct
sum of matrices of the
form
\[
\text{(a)}\
J_n(0),\quad
\text{(b)}\ [\Phi\diag
I_n]\text{ if
$\sqrt[T]{\Phi}$ does
not
exist},\quad\text{(c)}\
\sqrt[T]{\Phi}q(\Phi).
\]

Consider each of these
summands.

\emph{Summands }(a).
Theorem \ref{Theorem
5}(b) ensures that the
summands of the form
$J_n(0)$ are uniquely
determined by $A$,
which gives the
summands (i) of the
theorem.

\emph{Summands }(b).
By Lemma
\ref{lsdy1}(a),
$\sqrt[T]{\Phi}$ does
not exist if and only
if \eqref{fus} holds.
Theorem \ref{Theorem
5}(b) ensures that the
summands of the form
$[\Phi\diag I_n]$ are
uniquely determined by
$A$,  up to replacement
of $\Phi$ by $\Psi$
with $\chi_{\Psi}(x)
=\chi^{\vee}_{\Phi}(x)$.
This gives the
summands (ii).

\emph{Summands }(c).
Let $\Phi$ be a
nonsingular $n\times
n$ Frobenius block for
which $\sqrt[T]{\Phi}$
exists. Then by Lemma
\ref{lsdy1}(a)
\begin{equation}
\label{htr1}
p_{\Phi}(x) =
p_{\Phi}^{\vee}(x),\qquad
p_{\Phi}(x)\ne x +
(-1)^{n+1}.
\end{equation}
Consider the whole
group of summands of the form $\sqrt[T]{\Phi}q(\Phi)$ with the same $\Phi$:
\begin{equation}\label{kuy}
\sqrt[T]{\Phi}q_1(\Phi)
\oplus\dots\oplus
\sqrt[T]{\Phi}q_s(\Phi).
\end{equation}

Let first $\deg
p_{\Phi}(x)>1$. Then
the involution
$f(\kappa)^{\circ} :=
f(\kappa^{-1})$ on the
field $\mathbb
F(\kappa) = \mathbb
F[x]/p_{\Phi}(x)\mathbb
F[x]$ (see
\eqref{alft} and
\eqref{alfta}) is
nonidentity; otherwise
$\kappa=\kappa^{\circ}=\kappa^{-1}$,
$\kappa^2-1=0$,
$(x^2-1)|p_{\Phi}(x)$,
and hence
$p_{\Phi}(x)=x\pm 1$
since it is
irreducible. By Lemma
\ref{THt}(b), the
Hermitian form
\[q_1(\kappa)x_1^{\circ}y_1+\dots+
q_s(\kappa)x_s^{\circ}y_s\]
over $F(\kappa)$ is
equivalent to
$x_1^{\circ}y_1+\dots+
x_s^{\circ}y_s$. By
Theorem \ref{Theorem
5}(b), the matrix
\eqref{kuy} is
congruent to
$\sqrt[T]{\Phi}
\oplus\dots\oplus
\sqrt[T]{\Phi}$ and
the summands of the
form $\sqrt[T]{\Phi}$
with $\deg
p_{\Phi}(x)>1$ are
uniquely determined by
$A$. This gives the
summands (iii).

Let now
$p_{\Phi}(x)=x+c$.
Then by \eqref{htr1}
and \eqref{ksu}
$x+c=c^{-1}(cx+1)$,
$c=c^{-1}$, $c=\pm 1$.
The inequality in
\eqref{htr1} implies
\begin{equation}
\label{hdtr}
p_{\Phi}(x)= x +
(-1)^n.
\end{equation}
By
\cite[Eq.(70)]{h-s_bilin_anyf},
\begin{equation}\label{1x11}
\Gamma_n^{-T}\Gamma_n=\Upsilon_n:=
(-1)^{n+1}
\begin{bmatrix} 1&2&&\text{\raisebox{-6pt}
{\large\rm *}}
\\&1&\ddots&\\
&&\ddots&2\\
0 &&&1
\end{bmatrix}.
\end{equation}
Hence
$\Gamma_n=\sqrt[T]{\Upsilon_n}$
and $\Upsilon_n$ is
similar to
$J_n((-1)^{n+1})$,
which is similar to
$\Phi$ due to
\eqref{hdtr}. By
Theorem \ref{Theorem
5}(c) we can take
$\Upsilon_n$ instead
of $\Phi$ with
$p_{\Phi}(x)= x +
(-1)^n$ in Theorem
\ref{Theorem 5}(a,b).
The field $\mathbb
F(\kappa) = \mathbb
F[x]/p_{\Phi}(x)\mathbb
F[x]$ is $\mathbb F$
with the identity
involution; all
polynomials $q_i(x)$
in \eqref{kuy} are
some scalars
$a_i\in\mathbb F$. By
Lemma \ref{THt}(a),
the quadratic form
\[q_1(\kappa)x_1^2+\dots+
q_s(\kappa)x_s^2=
a_1x_1^2+\dots+
a_sx_s^2\] over
$\mathbb F$ is
equivalent to
\[
\text{either }\
x_1^2+\dots
+x_r^2,\qquad\text{or
}\ \zeta x_1^2+
x_2^2+\dots +x_r^2,
\]
in which $\zeta $ is a
fixed nonsquare of
$\mathbb F$. Theorem
\ref{Theorem 5}(b)
ensures that
\eqref{kuy} is
congruent to
\[\text{either }\
\Gamma_n
\oplus\dots\oplus
\Gamma_n,\qquad\text{or
}\ \zeta \Gamma_n
\oplus\Gamma_n
\oplus\dots\oplus
\Gamma_n,
\]
and this sum is
uniquely determined by
$A$.  This gives the
summands (iv).
\end{proof}

Note that if
$\sqrt[T]{\Phi}$
exists, then $\deg
p_{\Phi}(x)$ is even
or $p_{\Phi}(x)= x +
(-1)^{n}$. Indeed,
$p_{\Phi}(x)=
p_{\Phi}^{\vee}(x)$ by
\eqref{lbdr}. Let
$p_{\Phi}(x)=x^n+c_1x^{n-1}+
\dots+c_n$. Then
\begin{equation*}\label{frw}
x^n+c_1x^{n-1}+
\dots+c_n=c_n^{-1}(c_nx^n+
\dots+c_1x+1),
\end{equation*}
$c_n=c_n^{-1}$, and
$\theta :=c_n=\pm 1$.
If $n=2m+1$, then
\[
p_{\Phi}(x)=x^n+c_1x^{n-1}+
\dots+c_{m+1}x^{m+1}+
\theta
c_{m+1}x^{m}+\dots
+\theta
c_1x^{n-1}+\theta,
\]
and so
$p_{\Phi}(-\theta
)=0$. Since
$p_{\Phi}(x)$ is
irreducible,
$p_{\Phi}(x)=x+\theta
$. By the inequality
\eqref{4.adlw},
$p_{\Phi}(x)= x +
(-1)^{n}$.

JaJa Example: $\mathbb
F_5/(x^2+x+1)$.

\begin{theorem}
\label{Th5a} Let $\mathbb F$ be a
finite field of characteristic not $2$
with nonidentity involution. Every
square matrix over\/ $\mathbb F$ is
*congruent to a direct sum, uniquely
determined up to permutation of
summands, of matrices of the following
types:
\begin{itemize}
 \item [{\rm(i)}]
$J_n(0)$;

 \item [{\rm(ii)}]
$[\Phi\diag I_n]$, in
which $\Phi$ is an
$n\times n$
nonsingular Frobenius
block such that $
p_{\Phi}(x) \ne
p_{\Phi}^{\vee}(x)$  $($see \eqref{ksu}$)$
and $\Phi$ is
determined up to
replacement by the
Frobenius block $\Psi$
with $\chi_{\Psi}(x)
=\chi^{\vee}_{\Phi}(x)$;

 \item [{\rm(iii)}]
$\sqrt[\displaystyle
*]{\Phi}$, in which $\Phi$ is a nonsingular Frobenius
block  such that
$p_{\Phi}(x) =
p_{\Phi}^{\vee}(x)$.
\end{itemize}
\end{theorem}

\begin{proof} Let
$\mathbb F$ be a
finite field of
characteristic not $2$
with nonidentity
involution.
By Theorem
\ref{Theorem 5}(a), every
square matrix $A$ over
$\mathbb F$ is
*congruent to a direct
sum of matrices of the
form
\[
\text{(a)}\
J_n(0),\quad
\text{(b)}\ [\Phi\diag
I_n]\text{ if
$\sqrt[\displaystyle
*]{\Phi}$ does not
exist},\quad\text{(c)}\
\sqrt[\displaystyle
*]{\Phi}q(\Phi).
\]

Consider each of these
summands.

\emph{Summands }(a).
Theorem \ref{Theorem
5}(b) ensures that the
summands of the form
$J_n(0)$ are uniquely
determined by $A$,
which gives the
summands (i) of the
theorem.

\emph{Summands }(b).
By Lemma
\ref{lsdy1}(a),
$\sqrt[\displaystyle
*]{\Phi}$ does not
exist if and only if $
p_{\Phi}(x) \ne
p_{\Phi}^{\vee}(x)$.
Theorem \ref{Theorem
5}(b) ensures that the
summands of the form
$[\Phi\diag I_n]$ are
uniquely determined by
$A$,  up to replacement
of $\Phi$ by $\Psi$
with $\chi_{\Psi}(x)
=\chi^{\vee}_{\Phi}(x)$.
This gives the
summands (ii).

\emph{Summands }(c).
Let $\Phi$ be a
nonsingular $n\times
n$ Frobenius block for
which
$\sqrt[\displaystyle
*]{\Phi}$ exists; this
means that $
p_{\Phi}(x) =
p_{\Phi}^{\vee}(x)$.
Consider the whole
group of summands of the form $\sqrt[\displaystyle
*]{\Phi}q(\Phi)$ with the same $\Phi$:
\begin{equation}\label{kuys1}
\sqrt[\displaystyle
*]{\Phi}q_1(\Phi)
\oplus\dots\oplus
\sqrt[\displaystyle
*]{\Phi}q_s(\Phi).
\end{equation}
The involution
$f(\kappa)^{\circ} :=
f(\kappa^{-1})$ on the
field $\mathbb
F(\kappa) = \mathbb
F[x]/p_{\Phi}(x)\mathbb
F[x]$ (see
\eqref{alft} and
\eqref{alfta}) is
nonidentity since it
extends the
nonidentity involution
on $\mathbb F$. By
Lemma \ref{THt}(b),
the Hermitian form
\[q_1(\kappa)x_1^{\circ}y_1+\dots+
q_s(\kappa)x_s^{\circ}y_s\]
over $F(\kappa)$ is
equivalent to
$x_1^{\circ}y_1+\dots+
x_s^{\circ}y_s$, and
so by Theorem
\ref{Theorem 5}(b) the
matrix \eqref{kuys1}
is *congruent to
$\sqrt[\displaystyle
*]{\Phi}
\oplus\dots\oplus
\sqrt[\displaystyle
*]{\Phi}$ and this
direct sum is uniquely
determined by $A$.
This gives the
summands (iii).
\end{proof}

\subsection{Canonical pairs of
symmetric or
skew-symmetric
matrices}

In this section, we
give canonical
matrices of pairs
consisting of
symmetric or
skew-symmetric forms.
Canonical matrices of
pairs of
skew-symmetric forms
are given in Corollary
\ref{Theorem 4}; it
remains to consider
pairs, in which the
first form is
symmetric and the
second is symmetric or
skew-symmetric.

For square matrices
$A,B,C,D$ of the same
size, we write
\[
(A,B)\oplus
(C,D)=(A\oplus
C,B\oplus D),\qquad
(A,B)C = (AC,BC).
\]

\begin{theorem}\label{Theo4a}
Each pair of symmetric matrices of the
same size over a finite field $\mathbb
F$ of characteristic not $2$ is
congruent to a direct sum that is
uniquely determined up to permutation
of summands and consists of any number
of summands of the following types:
\begin{itemize}
  \item[\rm(i)]
$([F_n\diag
 F_n^T],\,
[G_n\diag G_n^T])$,
where $F_n$ and $G_n$
are defined in
\eqref{ser17};

  \item[\rm(ii)]
for each nonsingular
Frobenius block
${\Phi}$:
\begin{itemize}
  \item[$\bullet$]
$(\Phi_{1},
\Phi_{1}\Phi)$, in
which $\Phi_1$ is
defined in
\eqref{vr2},

  \item[$\bullet$]
at most one summand $
(\Phi_{1},
\Phi_{1}\Phi)
f_{\Phi}(\Phi)$, in
which $f_{\Phi}(x)\in
\mathbb F[x]$ is a
fixed $($for each
$\Phi$$)$ polynomial
of degree
$<\deg(p_{\Phi}(x))$
such that
\[f_{\Phi}(\omega)\in{\mathbb
F}(\omega):={\mathbb
F}[x]/p_{\Phi}(x){\mathbb
F}[x]\] is not a
square;
\end{itemize}

  \item[\rm(iii)]
for each
$n=1,2,\dots$:
\begin{itemize}
  \item[$\bullet$]
the pair of
$n\times n$ matrices
\begin{equation}\label{ser2r}
B_n:= \left(\!
 \begin{bmatrix}
0&&1&0\\
&\ddd&\ddd\\
1&0\\
0&&&0
\end{bmatrix},
\begin{bmatrix}
0&&&1\\
&&\ddd\\
&1\\1&&&0
\end{bmatrix}
 \!\right),
\end{equation}

\item[$\bullet$] at
most one summand
$\zeta B_n$, in which
$\zeta$ is a fixed
nonsquare of $\mathbb
F$.
\end{itemize}
\end{itemize}
\end{theorem}

\begin{proof}
Let $\mathbb F$ be a
finite field of
characteristic not $2$
with the identity
involution. By Lemma
\ref{THEjM}, the
matrix $\Phi_1$ exists
for each nonsingular
Frobenius block $\Phi$
over $\mathbb F$. By
Theorem \ref{The4}(a),
each pair $(A,B)$ of
symmetric matrices of
the same size is
congruent to a direct
sum of pairs of the
form
\[
\text{(a)}\ ([F_n\diag
F_n^T],\, [G_n\diag
G_n^T]),\quad\text{(b)}\
A_{\Phi}^{f(x)}:=
(\Phi_1,
\Phi_1\Phi)f(\Phi),
\quad\text{(c)}\
B_n^a,
\]
in which $f(x)\in
\mathbb F[x]$ is a
nonzero polynomial of
degree
$<\deg(p_{\Phi}(x))$
and $0\ne a\in \mathbb F$.

Consider each of these
summands.

\emph{Summands }(a).
Theorem \ref{The4}(b)
ensures that the
summands of the form
(a) are uniquely
determined by $(A,B)$,
which gives the
summands (i) of the
theorem.

\emph{Summands }(b).
Consider the whole
group of summands of the form $A_{\Phi}^{g(x)}$ with the same
nonsingular Frobenius
block $\Phi$:
\begin{equation}\label{kuy3n}
A_{\Phi}^{g_1(x)}
\oplus\dots\oplus
  A_{\Phi}^{g_s(x)}.
\end{equation}
By Lemma
\ref{THt}(a), the
quadratic form
\[q_1(\omega)x_1^2+\dots+
q_s(\omega)x_s^2\]
over ${ \mathbb
F}(\omega)={ \mathbb
F}[x]/p_{\Phi}(x){\mathbb
F}[x]$ is equivalent
to
\[
\text{either }\
x_1^2+\dots
+x_r^2,\qquad\text{or
}\ f_{\Phi}(\omega)
x_1^2+ x_2^2+\dots
+x_r^2,
\]
in which
$f_{\Phi}(x)\in
\mathbb F[x]$ is a
fixed nonzero polynomial of
degree
$<\deg(p_{\Phi}(x))$
such that
$f_{\Phi}(\omega)\in{\mathbb
F}(\omega)$ is not a
square. Theorem
\ref{Theorem 5}(b)
ensures that
\eqref{kuy3n} is
congruent to
\[\text{either }\
A_{\Phi}^1
\oplus\dots\oplus
A_{\Phi}^1,\qquad\text{or
}\
A_{\Phi}^{f_{\Phi}(x)}
\oplus A_{\Phi}^1
\oplus\dots\oplus
A_{\Phi}^1,
\] and this sum is uniquely determined by $(A,B)$.
This gives the
summands (ii).

\emph{Summands }(c).
Consider the whole
group of summands of the form $B_n^{a}$ with the same $n$:
\begin{equation}\label{kuy3e}
B_n^{a_1}
\oplus\dots\oplus
 B_n^{a_s}.
\end{equation}
By Lemma \ref{THt}(a),
the quadratic form
$a_1x_1^2+\dots+
a_sx_s^2$ over
${\mathbb F}$ is
equivalent to either
$x_1^2+\dots +x_r^2$,
or $\zeta x_1^2+
x_2^2+\dots +x_r^2$,
in which $\zeta$ is a
fixed nonsquare of
$\mathbb F$. Theorem
\ref{Theorem 5}(b)
ensures that
\eqref{kuy3e} is
congruent to
\[\text{either }\
B_n^{1}
\oplus\dots\oplus
B_n^{1},\qquad\text{or
}\ B_n^{\zeta} \oplus
B_n^{1}
\oplus\dots\oplus
B_n^{1},
\] and this sum is uniquely determined by $(A,B)$.
This gives the
summands (iii).
\end{proof}

\begin{theorem}\label{Theorem 4fbs}
Each pair consisting of a symmetric
matrix and a skew-symmetric matrix of
the same size over a finite field
$\mathbb F$ of characteristic not $2$
is congruent to a direct sum  that is
uniquely determined up to permutation
of summands and consists of any number
of summands of the following types:
\begin{itemize}
  \item[\rm(i)]
$([F_n\diag
 F_n^T],\,
[G_n\diag - G_n^T])$,
in which $F_n$ and
$G_n$ are defined in
\eqref{ser17};

  \item[\rm(ii)]
$([I_n\diag I_n],\,
[\Phi\diag -
\Phi^T])$, in which
$\Phi$ is an $n\times
n$ Frobenius block
such that
\begin{equation}\label{yte}
p_{\Phi}(x)\notin\mathbb
F[x^2],\qquad\Phi\ne
J_1(0),J_3(0),J_5(0),\dots
\end{equation}
$($see \eqref{fik}$)$,
and $\Phi$ is
determined up to
replacement by the
Frobenius block $\Psi$
with $\chi_{\Psi}(x)
=(-1)^{\det
\chi_{\Phi}}\chi_{\Phi}(-
x)$;

  \item[\rm(iii)]
$(\Phi_{-1},
\Phi_{-1}\Phi)$, in
which ${\Phi}$ is a
Frobenius block such
that
$p_{\Phi}(x)\in\mathbb
F[x^2]$;

  \item[\rm(iv)]
$([J_n(0)\diag
J_n(0)^T], [I_n\diag
-I_n])$, in which $n$
is odd;

  \item[\rm(v)]
for each
$n=1,2,3,\dots$:
\begin{itemize}
  \item[$\bullet$]
the pair
of $n$-by-$n$
symmetric and
skew-symmetric
matrices defined as follows:
\begin{equation}\label{seh}
C_n:=\left(\!
 \begin{bmatrix}
0&&1\\
&\ddd\\
1&&0
\end{bmatrix},
\begin{bmatrix}
0&&&&&1&0\\
&&&&\ddd&\ddd\\
&&&1&0&\\
&&-1&0&&\\
&\ddd&0&&&\\
-1&\ddd&&&&\\
0&&&&&&0
\end{bmatrix}
 \!\right)
\end{equation}
if $n$ is odd, and
\begin{equation}\label{sehe}
C_n:=\left(\!
 \begin{bmatrix}
0&&1&0\\
&\ddd&\ddd\\
1&0\\
0&&&0
\end{bmatrix},
\begin{bmatrix}
0&&&&&1\\
&&&&\ddd&\\
&&&1&&\\
&&-1&&&\\
&\ddd&&&&\\
-1&&&&&0\\
\end{bmatrix}
 \!\right)
\end{equation}
if $n$ is even,

\item[$\bullet$] at
most one summand of
the form $\zeta C_n$,
in which $\zeta$ is a
fixed nonsquare of
$\mathbb F$.
\end{itemize}
\end{itemize}
\end{theorem}

\begin{proof}
Let $\mathbb F$ be a
finite field of
characteristic not $2$
with the identity
involution. By Theorem
\ref{The4}(a) and
Remark \ref{hts}, each
pair $(A,B)$
consisting of a
symmetric matrix $A$
and a skew-symmetric
matrix $B$ of the same
size is congruent to a
direct sum of pairs of
the form
\begin{itemize}
  \item[(a)]
$([F_n\diag F_n^T],\,
[G_n\diag -G_n^T])$,

  \item[(b)] $([I_n\diag
 I_n],\,
[\Phi\diag - \Phi^T])$
if $\Phi_{-1}$ does
not exist,

  \item[(c)]
$A_{\Phi}^{f(x)}:=
(\Phi_{-1},
\Phi_{-1}\Phi)f(\Phi)$,
in which $0\ne f(x)=
f(- x)\in \mathbb
F[x]$ and
$\deg(f(x))<\deg(p_{\Phi}(x))$,

 \item[(d)]
$([J_n(0)\diag
J_n(0)^T], [I_n\diag
-I_n])$, in which $n$
is odd,

 \item[(e)]
$ C_n^a$ (defined in
\eqref{serk1}), in
which $n$
is even and $0\ne a\in \mathbb
F$.
\end{itemize}

Consider each of these
summands.

\emph{Summands }(a).
Theorem \ref{The4}(b)
ensures that the
summands
(a) are uniquely
determined by $(A,B)$,
which gives the
summands (i) of the
theorem.

\emph{Summands }(b).
By Lemma \ref{THEjM},
$\Phi_{-1}$ does not
exist if and only if
\eqref{yte} is
satisfied. Theorem
\ref{The4}(b) ensures
that the summands (b) are
uniquely determined by
$(A,B)$, up to
replacement of $\Phi$
by $\Psi$ with
$\chi_{\Psi}(x)
=(-1)^{\det
\chi_{\Phi}}\chi_{\Phi}(-
x)$, which gives the
summands (ii).

\emph{Summands }(c).
By Lemma \ref{THEjM},
$\Phi_{-1}$ exists if
and only if
\eqref{yte} is not
satisfied; that is,
\begin{equation}\label{yteq}
p_{\Phi}(x)\in\mathbb
F[x^2]\quad\text{or}\quad
\Phi=
J_1(0),J_3(0),J_5(0),\dots
\end{equation}
Consider the whole
group of summands of the form $A_{\Phi}^{f(x)}$ with the same
nonsingular Frobenius
block $\Phi$:
\begin{equation}\label{kuy3s}
A_{\Phi}^{f_1(x)}
\oplus\dots\oplus
  A_{\Phi}^{f_s(x)}.
\end{equation}

Let first
$p_{\Phi}(x)\in\mathbb
F[x^2]$. Then the
involution
$f(\omega)^{\circ}=
f(-\omega)$ on ${
\mathbb F}(\omega)={
\mathbb
F}[x]/p_{\Phi}(x){\mathbb
F}[x]$  is nonidentity
(since $
\omega^{\circ}=-\omega\ne
\omega$). By Lemma
\ref{THt}(b), the
Hermitian form
\[f_1(\omega)x_1^{\circ}y_1+\dots+
f_s(\omega)x_s^{\circ}y_s\]
over ${ \mathbb
F}(\omega)$ is
equivalent to $
x_1^{\circ}y_1+\dots
+x_r^{\circ}y_r$.
Theorem \ref{Theorem
5}(b) ensures that
\eqref{kuy3s} is
congruent to
$A_{\Phi}^1
\oplus\dots\oplus
A_{\Phi}^1$, which
gives the summands
(iii).

Let now $\Phi= J_n(0)$
with $n=2m+1$ and
$m=1,2,\dots$. The
equalities \eqref{vr2}
hold for the $n\times
n$ matrices
\begin{equation}\label{jtw}
\Phi':=\begin{bmatrix}
0&&&&&&0
\\
-1&\ddots&&&&\ddd\\
&\ddots&0&0&0&\\
&&-1&0&0&\\
&&0&1&0&\\
&\ddd&&&\ddots&\ddots\\
0&&&&&1&0
\end{bmatrix},\qquad \Phi'_{-1}:=
\begin{bmatrix}
0&&1\\
&\ddd\\
1&&0
\end{bmatrix}
\end{equation}
instead of $\Phi$ and
$ \Phi_{-1}$. Since
$\Phi$ and $\Phi'$
are similar, by
Theorem \ref{The4}(c)
we can take
$C_n^{f(x)}:=(\Phi'_{-1},
\Phi'_{-1}\Phi')f(\Phi')$
instead of
$A_{\Phi}^{f(x)}$ and
\begin{equation}\label{kuj}
C_n^{f_1(x)}
\oplus\dots\oplus
  C_n^{f_s(x)}
\end{equation}
instead of
\eqref{kuy3s}.

Since $p_{\Phi}(x)=x$, the field $\mathbb
F(\omega) = \mathbb
F[x]/p_{\Phi}(x)\mathbb
F[x]$ is $\mathbb F$
with the identity
involution and all
polynomials $f_i(x)$
in \eqref{kuj} are
some scalars
$a_i\in\mathbb F$. By
Lemma \ref{THt}(a),
the quadratic form
$a_1x_1^2+\dots+
a_sx_s^2$ over
${\mathbb F}$ is
equivalent to either
$x_1^2+\dots +x_r^2$,
or $\zeta x_1^2+
x_2^2+\dots +x_r^2$,
in which $\zeta$ is a
fixed nonsquare of
$\mathbb F$. Theorem
\ref{Theorem 5}(b)
ensures that
\eqref{kuy3s} is
congruent to
\begin{equation}\label{trs}
\text{either }\ C_n^1
\oplus\dots\oplus
C_n^1,\qquad\text{or
}\ C_n^{\zeta } \oplus
C_n^1
\oplus\dots\oplus
C_n^1,
\end{equation}
and this sum is
uniquely determined by
$(A,B)$. This gives
the summands (v) with
odd $n$.

\emph{Summands }(d).
Theorem \ref{The4}(b)
ensures that the
summands of the form
(d) are uniquely
determined by $(A,B)$,
which gives the
summands (iv).

\emph{Summands }(e).
Consider the whole
group of summands of
the form $C_n^a$ with the same $n$:
\begin{equation}\label{kur}
C_n^{a_1}
\oplus\dots\oplus
 C_n^{a_s}.
\end{equation}
By Lemma \ref{THt}(a),
the quadratic form
$a_1x_1^2+\dots+
a_sx_s^2$ over
${\mathbb F}$ is
equivalent to either
$x_1^2+\dots +x_r^2$,
or $\zeta x_1^2+
x_2^2+\dots +x_r^2$,
in which $\zeta$ is a
fixed nonsquare of
$\mathbb F$. Theorem
\ref{Theorem 5}(b) and
Remark \ref{hts}
ensure that
\eqref{kur} is
congruent to
\[\text{either }\
C_n^{1}
\oplus\dots\oplus
C_n^{1},\qquad\text{or
}\ C_n^{\zeta } \oplus
C_n^{1}
\oplus\dots\oplus
C_n^{1},
\]
and this sum is
uniquely determined by
$(A,B)$, which gives
the summands (v) with
even $n$.
\end{proof}

\subsection{Canonical pairs of
Hermitian matrices}

\begin{theorem}\label{Thk4a}
Let $\mathbb F$ be a
finite field of
characteristic not $2$
with nonidentity
involution. Let $\mathbb
F_{\circ}$ be the
fixed field of
$\mathbb F$.  Each
pair of Hermitian
matrices of the same
size over $\mathbb F$
is *congruent to a
direct sum, uniquely
determined up to
permutation of
summands, of pairs of
the  following types:
\begin{itemize}
  \item[\rm(i)]
$([F_n\diag
 F_n^*],\,
[G_n\diag G_n^*])$,
in which $F_n$ and $G_n$
are defined in
\eqref{ser17};

  \item[\rm(ii)]
$([I_n\diag I_n],\,
[\Phi\diag \Phi^*])$,
in which $\Phi$ is an
$n\times n$ Frobenius
block over $\mathbb F$
such that $
p_{\Phi}(x)\notin\mathbb
F_{\circ}[x]$, and
$\Phi$ is determined
up to replacement by
the Frobenius block
$\Psi$ with
$\chi_{\Psi}(x)
=\bar\chi_{\Phi}(x)$ $($see \eqref{iut}$)$;

  \item[\rm(iii)]
$(\Phi_{1},
\Phi_{1}\Phi)$, in
which $\Phi$ is a
Frobenius block over
$\mathbb F_{\circ}$;

  \item[\rm(iv)] the pair of $n\times n$ matrices
\begin{equation}\label{ser2r1}
B_n:= \left(\!
 \begin{bmatrix}
0&&1&0\\
&\ddd&\ddd\\
1&0\\
0&&&0
\end{bmatrix},
\begin{bmatrix}
0&&&1\\
&&\ddd\\
&1\\1&&&0
\end{bmatrix}
 \!\right),\quad n=1,2,\dots
\end{equation}
\end{itemize}
\end{theorem}

\begin{proof}
Let $\mathbb F$ be a
finite field of
characteristic not $2$
with nonidentity
involution.
By Theorem
\ref{The4}(a), each pair
$(A,B)$ of Hermitian
matrices over $\mathbb F$ of the same
size is *congruent to
a direct sum of pairs
of the form
\begin{itemize}
  \item[\rm(a)]
$([F_n\diag F_n^*],\,
[G_n\diag G_n^*])$,

  \item[\rm(b)]
$([I_n\diag I_n],\,
[\Phi\diag \Phi^*])$
if $\Phi_1$ does not
exist,

  \item[\rm(c)]
$A_{\Phi}^{f(x)}:=
(\Phi_{\delta},
\Phi_1\Phi)f(\Phi)$,
in which $0\ne
f(x)=\bar f(x)\in
\mathbb F[x]$ and
$\deg(f(x))<\deg(p_{\Phi}(x))$,

  \item[\rm(d)] $B_n^a$ (defined in
\eqref{ser120}), in
which $0\ne a=\bar
a\in \mathbb F$.
\end{itemize}

Consider each of these
summands.

\emph{Summands }(a).
Theorem \ref{The4}(b)
ensures that the
summands of the form
(a) are uniquely
determined by $(A,B)$,
which gives the
summands (i) of the
theorem.

\emph{Summands }(b).
By Lemma
\ref{THEjM}(a),
${\Phi}_1$ does not
exist if and only if $
p_{\Phi}(x) \ne \bar
p_{\Phi} (x)$; that
is, $
p_{\Phi}(x)\notin\mathbb
F_{\circ}[x]$. Theorem
\ref{The4}(b) ensures
that the summands of
the form (b) are
uniquely determined,
up to replacement of
$\Phi$ by $\Psi$ with
$\chi_{\Psi}(x)
=\bar\chi_{\Phi}(x)$.
This gives the
summands (ii).

\emph{Summands }(c).
Consider the whole
group of summands of the form $A_{\Phi}^{f(x)}$ with the same
nonsingular Frobenius
block $\Phi$:
\begin{equation}\label{kuy3w}
A_{\Phi}^{f_1(x)}
\oplus\dots\oplus
  A_{\Phi}^{f_s(x)}.
\end{equation}
By Lemma
\ref{THt}(b), the
Hermitian form
\[f_1(\omega)x_1^{\circ}y_1+\dots+
f_s(\omega)x_s^{\circ}y_s
\]
over ${ \mathbb
F}(\omega)={ \mathbb
F}[x]/p_{\Phi}(x){\mathbb
F}[x]$ with involution
$f(\omega)^{\circ}=
\bar f(\omega)$ is
equivalent to
$x_1^{\circ}x_1+\dots+
x_s^{\circ}x_s$.
Theorem \ref{The4}(b)
ensures that
\eqref{kuy3w} is
*congruent to
$A_{\Phi}^1
\oplus\dots\oplus
A_{\Phi}^1$ and this
sum is uniquely
determined by $(A,B)$.
This gives the
summands (iii).

\emph{Summands }(d).
Consider the whole
group of summands of
the form $B_n^a$ with the same $n$:
\begin{equation}\label{kuy3c}
B_n^{a_1}
\oplus\dots\oplus
 B_n^{a_s}.
\end{equation}
By Lemma \ref{THt}(b),
the Hermitian form
$a_1\bar x_1y_1+\dots+
a_s\bar x_sy_s$
over ${\mathbb F}$ is
equivalent to
$\bar x_1y_1+\dots+
\bar x_sy_s$.
Theorem \ref{Theorem
5}(b) ensures that
\eqref{kuy3c} is
*congruent to $B_n^{1}
\oplus\dots\oplus
B_n^{1}$ and this sum
is uniquely determined
by $(A,B)$. This gives
the summands (iv).

\end{proof}

\section{Canonical
forms over $\mathfrak p$-adic fields
}\label{loc}

In this section
$\mathbb K$ is a
finite extension of
$\mathbb Q_p$ with
$p\ne 2$.

Let us recall the
classification of
quadratic and
Hermitian forms over
$\mathbb K$.
Each nonzero element
of $\mathbb Q_p$ can be
represented in the
form
\begin{equation}\label{ypf}
a=\alpha _zp^z+\alpha _{z+1}p^{z+1}+
\cdots,\qquad z\in \mathbb Z,\quad \alpha _z\ne 0,
\end{equation}
in which all
$\alpha _i\in\{0,1,\dots,p-1\}$.
The \emph{exponential
variation} on $\mathbb
K$ is the following mapping $\nu: \mathbb
K\to \mathbb R\cup\{+
\infty\}$:
\[
\nu(a):=
  \begin{cases}
   +\infty, & \text{if }a=0, \\
    z, & \text{if $0\ne a\in \mathbb
Q_p$ is represented in the form \eqref{ypf}},\\
\nu(\alpha_m)/m, &
\text{if $a\notin
\mathbb Q_p$ and its minimum polynomial of
over $\mathbb
Q_p$ is}\\&
\qquad \qquad x^m+\alpha
_1x^{m-1}+\dots+\alpha
_m.
  \end{cases}
\]

The ring
\begin{equation}
{\cal O}(\mathbb
K):=\{a\in\mathbb
K\,|\,\nu (a)\ge 0\}
\end{equation}
is called the
\emph{ring of
integers} of $\mathbb
K$; it is a principal
ideal ring, whose
unique maximal ideal
is
\begin{equation}
\mathfrak
m:=\{a\in\mathbb
K\,|\,\nu (a)> 0\}=\pi
{\cal O}(\mathbb K).
\end{equation}
Each generator $\pi$
of $\mathfrak m$ is
called a \emph{prime
element}. The set
\begin{equation}
{\cal O}(\mathbb
K)^{\times}:=\{a\in\mathbb
K\,|\,\nu (a)= 0\}
\end{equation}
is the group of all invertible elements
of ${\cal O}(\mathbb K)$; they are
called the \emph{units} of $\mathbb K$.

The factor ring
\begin{equation}
{\cal O}(\mathbb
K)/\mathfrak m
\end{equation}
is a field, which is
called the
\emph{residue field}
of $\mathbb K$; it is
a finite extension of
the residue field
$\mathbb F_p=\mathbb
Q_p/p\mathbb Q_p$ of
$\mathbb Q_p$.

\begin{lemma}\label{kux}
Let $\mathbb K$ be a
finite extension of
$\mathbb Q_p$ with
$p\ne 2$. Let its
residue field ${\cal
O}(\mathbb
K)/\mathfrak m$
consist of $p^m$
elements. Let $u\in
{\cal O}(\mathbb
K)^{\times}\setminus
\mathbb K^{\times 2}$
be a unit that is not
a square, and $\pi$ be
a prime element. Then
each quadratic form of
rank $r\ge 1$ over
$\mathbb K$ is
equivalent to exactly
one form
\begin{equation}\label{ghts}
c_1x_1^2+c_2x_2^2+\dots+c_tx_t^2
+x_{t+1}^2+\dots+x_r^2,
\end{equation}
in which
$(c_1,\dots,c_t)$ is
one of the sequences:
\begin{equation}\label{gjrs}
(1),\ (u),\ (\pi),\
(u\pi),\ (u,\pi),\
(u,u\pi),\
(\pi,r\pi),\
(u,\pi,r\pi),
\end{equation}
where
\begin{equation}
r:=\begin{cases}
    u & \text{if $p^m\equiv 1\mod 4$}, \\
    1 & \text{if $p^m\equiv 3\mod 4$}.
  \end{cases}
\end{equation}
\end{lemma}

\begin{lemma}\label{kurt}
Let a field $\mathbb K$ with nonidentity
involution be a
finite extension of
$\mathbb Q_p$, $p\ne
2$. Let
$\mathbb K_{\circ}$ be
the fixed field with
respect to this
involution. Let $u\in
{\cal O}(\mathbb
K_{\circ})^{\times}\setminus
\mathbb
K_{\circ}^{\times 2}$
be a unit that is not
a square, and $\pi$ be
a prime element of
$\mathbb K_{\circ}$.
Then each Hermitian
form with nonzero
determinant over
$\mathbb K$ is
classified by
dimension and
determinant; moreover,
it is equivalent to
\begin{equation*}\label{tqrs}
\text{either }\ \bar x_1y_1+\dots +\bar x_n y_n,\qquad\text{or
}\ t\bar x_1y_1+\bar x_2y_2+\dots +\bar x_n y_n,
\end{equation*}
in which
\begin{equation}\label{fsr}
t:=\begin{cases} \pi &
\text{if $\mathbb
K=\mathbb K_{\circ}(\sqrt u)$,}\\
u& \text{if $\mathbb
K=\mathbb
K_{\circ}(\sqrt \pi)$
or $\mathbb
K_{\circ}(\sqrt
{u\pi})$}.
  \end{cases}
\end{equation}
\end{lemma}

\begin{proof}
 By
\cite[Ch. 10, Example
1.6(ii)]{sch}, regular
Hermitian forms over
$\mathbb K$ are
classified by
dimension and
determinant.
\end{proof}

\subsection{Canonical matrices for
congruence and
*congruence}

\begin{theorem}
\label{Theorem 5f} Let a field $\mathbb
F$ be a finite extension of $\mathbb
Q_p$ with $p\ne 2$. Every square matrix
over\/ $\mathbb F$ is congruent to a
direct sum  that is uniquely determined
up to permutation of summands and
consists of any number of summands of
the following types:
\begin{itemize}
 \item [{\rm(i)}]
$J_n(0)$;

 \item [{\rm(ii)}]
$[\Phi\diag I_n]$, in
which $\Phi$ is an
$n\times n$
nonsingular Frobenius
block over\/ $\mathbb
F$ such that
\begin{equation}\label{funs2}
p_{\Phi}(x) \ne
p_{\Phi}^{\vee}(x)\ (\text{see \eqref{ksu1}})\qquad
\text{or}\qquad
p_{\Phi}(x)= x +
(-1)^{n+1}
\end{equation}
and $\Phi$ is
determined up to
replacement by the
Frobenius block $\Psi$
with $\chi_{\Psi}(x)
=\chi^{\vee}_{\Phi}(x)$;

 \item [{\rm(iii)}]
for each nonsingular
Frobenius block $\Phi$
over\/ $\mathbb F$
such that $p_{\Phi}(x)
= p_{\Phi}^{\vee}(x)$
and $\deg
p_{\Phi}(x)\ge 2$:
\begin{itemize}
  \item[$\bullet$]
$\sqrt[T]{\Phi}$,
  \item[$\bullet$]
at most one summand of
the form
\[
  \begin{cases}
\sqrt[T]{\Phi} \tilde{
\pi} (\Phi) & \text{if
$\mathbb K=\mathbb
K_{\circ}
(\sqrt { u})$,}\\
\sqrt[T]{\Phi}\tilde{
u}(\Phi) & \text{if
$\mathbb K=\mathbb
K_{\circ}(\sqrt {
\pi})$ or $\mathbb K=
\mathbb
K_{\circ}(\sqrt {{ u}{
\pi}})$},
  \end{cases}
\]
in which $\mathbb K$ is
the following field
with involution:
\[
\mathbb K:= \mathbb
F(\kappa)=\mathbb
F[x]/p_{\Phi}(x)\mathbb
F[x],\qquad
f(\kappa)^{\circ} :=
f(\kappa^{-1})\]
$($see \eqref{alft}
and \eqref{alfta}$)$,
$\mathbb K_{\circ}$ is
its fixed field, ${
\pi}$ is a prime
element of $\mathbb
K_{\circ}$, ${ u}\in
{\cal O}(\mathbb
K_{\circ})^{\times}
\setminus \mathbb
K_{\circ}^{\times 2}$
is a unit that is not
a square, $\tilde{
\pi}(x),\tilde{
u}(x)\in \mathbb
F[x,x^{-1}]$ are the
functions of the form
\eqref{ser13} such
that $\tilde{
\pi}(\kappa)={ \pi}$
and $\tilde{
u}(\kappa)={ u}$;
\end{itemize}

 \item [{\rm(iv)}]
for each
$n=1,2,\dots$:
\begin{itemize}
  \item[$\bullet$]
$\Gamma_{n}$
$($defined in
\eqref{kus}$)$,
  \item[$\bullet$] at most one summand
from the list
\begin{gather*}\label{grsk}
u \Gamma_{n},\
\pi\Gamma_{n},\ u
\pi\Gamma_{n},\ u
\Gamma_{n}\oplus
\pi\Gamma_{n},\ u
\Gamma_{n}\oplus u
\pi\Gamma_{n},\\
\pi\Gamma_{n}\oplus r
\pi\Gamma_{n},\ u
\Gamma_{n}\oplus
\pi\Gamma_{n}\oplus r
\pi\Gamma_{n},
\end{gather*}
in which
\begin{equation}\label{feak}
r:=
\begin{cases}
u &
\text{if $p^m\equiv 1\mod 4$}, \\
1 & \text{if
$p^m\equiv 3\mod 4$},
  \end{cases}
\end{equation}
$p^m$ is the number of
elements of the
residue field ${\cal
O}(\mathbb
F)/\mathfrak m$ of
$\mathbb F$, $u\in
{\cal O}(\mathbb
F)^{\times}\setminus
\mathbb F^{\times 2}$
is a unit that is not
a square, and $\pi$ is
a prime element of
$\mathbb F$.
\end{itemize}
\end{itemize}
\end{theorem}

\begin{proof}
Let
a field $\mathbb F$ be
a finite extension of
$\mathbb Q_p$ with
$p\ne 2$.
By Theorem
\ref{Theorem 5}(a),
every square matrix
$A$ over $\mathbb F$
is congruent to a
direct sum of matrices
of the form
\[
\text{(a)}\
J_n(0),\quad
\text{(b)}\ [\Phi\diag
I_n]\text{ if
$\sqrt[T]{\Phi}$ does
not
exist},\quad\text{(c)}\
\sqrt[T]{\Phi}q(\Phi).
\]

Consider each of these
summands.

\emph{Summands }(a).
Theorem \ref{Theorem
5}(b) ensures that the
summands of the form
$J_n(0)$ are uniquely
determined by $A$,
which gives the
summands (i) of the
theorem.

\emph{Summands }(b).
By Lemma
\ref{lsdy1}(a),
$\sqrt[T]{\Phi}$ does
not exist if and only
if \eqref{funs2}
holds. Theorem
\ref{Theorem 5}(b)
ensures that the
summands of the form
$[\Phi\diag I_n]$ are
uniquely determined by
$A$,  up to
replacement of $\Phi$
by $\Psi$ with
$\chi_{\Psi}(x)
=\chi^{\vee}_{\Phi}(x)$.
This gives the
summands (ii).

\emph{Summands }(c).
Let $\Phi$ be a
nonsingular $n\times
n$ Frobenius block for
which $\sqrt[T]{\Phi}$
exists. Then by Lemma
\ref{lsdy1}(a)
\begin{equation}
\label{htr}
p_{\Phi}(x) =
p_{\Phi}^{\vee}(x),\qquad
p_{\Phi}(x)\ne x +
(-1)^{n+1}.
\end{equation}
Consider the whole
group of summands of the form $\sqrt[T]{\Phi}q(\Phi)$ with the same $\Phi$:
\begin{equation}\label{kuyp}
\sqrt[T]{\Phi}q_1(\Phi)
\oplus\dots\oplus
\sqrt[T]{\Phi}q_s(\Phi).
\end{equation}

Let first $\deg
p_{\Phi}(x)>1$. Then
the involution
$f(\kappa)^{\circ} :=
f(\kappa^{-1})$ on the
field $\mathbb
F(\kappa) = \mathbb
F[x]/p_{\Phi}(x)\mathbb
F[x]$ (see
\eqref{alft} and
\eqref{alfta}) is
nonidentity; otherwise
$\kappa=\kappa^{\circ}=\kappa^{-1}$,
$\kappa^2-1=0$,
$x^2-1$ divides $p_{\Phi}(x)$,
and hence
$p_{\Phi}(x)=x\pm 1$
since it is
irreducible. By Lemma
\ref{kurt}, the
Hermitian form
\[q_1(\kappa)x_1^{\circ}y_1+\dots+
q_s(\kappa)x_s^{\circ}y_s\]
over $\mathbb
F(\kappa)$ is
equivalent to either $x_1^{\circ}y_1+\dots+
x_s^{\circ}y_s$ or
$tx_1^{\circ}y_1+x_2^{\circ}y_2+\dots+
x_s^{\circ}y_s$, in
which $t$ is defined
in \eqref{fsr}.
Theorem \ref{Theorem
5}(b) ensures that the
matrix \eqref{kuyp} is
congruent to
\[
\text{either }\ \sqrt[T]{\Phi}
\oplus\dots\oplus
\sqrt[T]{\Phi},\qquad\text{or
}\ \sqrt[T]{\Phi}\tilde
t(\Phi)
\oplus\sqrt[T]{\Phi}
\oplus\dots\oplus
\sqrt[T]{\Phi},
\]
in which $\tilde t(x)\in
\mathbb F[x,x^{-1}]$
is the function of the
form \eqref{ser13}
such that
$\tilde{t}(\kappa)=t$.
This sum is uniquely
determined by $A$,
which gives the
summands (iii).

Let now
$p_{\Phi}(x)=x+c$.
Then by \eqref{htr}
and \eqref{ksu},
$x+c=c^{-1}(cx+1)$,
$c=c^{-1}$, $c=\pm 1$.
The inequality in
\eqref{htr} implies
\begin{equation}
\label{hdtr1}
p_{\Phi}(x)= x +
(-1)^n.
\end{equation}
By \eqref{1x11},
$\Gamma_n=\sqrt[T]{\Upsilon_n}$
and $\Upsilon_n$ is
similar to
$J_n((-1)^{n+1})$,
which is similar to
$\Phi$ due to
\eqref{hdtr}. By
Theorem \ref{Theorem
5}(c) we can take
$\Upsilon_n$ instead
of $\Phi$ with
$p_{\Phi}(x)= x +
(-1)^n$ in Theorem
\ref{Theorem 5}(a,b).
The field $\mathbb
F(\kappa) = \mathbb
F[x]/p_{\Phi}(x)\mathbb
F[x]$ is $\mathbb F$
with the identity
involution; all
polynomials $q_i(x)$
in \eqref{kuyp} are
some scalars
$a_i\in\mathbb F$, by
Lemma \ref{kux} with
$\mathbb K= \mathbb
F$, the quadratic form
\[q_1(\kappa)x_1^2+\dots+
q_s(\kappa)x_s^2=
a_1x_1^2+\dots+
a_sx_s^2\] over
$\mathbb F$ is
equivalent to exactly
one form \eqref{ghts},
in which
$(c_1,\dots,c_t)$ is
one of the sequences
\eqref{gjrs}. Theorem
\ref{Theorem 5}(b)
ensures that
\eqref{kuyp} is
congruent to a direct
sum of matrices of the
form (iv), and this
sum is uniquely
determined by $A$.
This gives the
summands (iv).
\end{proof}

\begin{theorem}
\label{Th 5} Let a
field $\mathbb F$ with
nonidentity involution
be a
finite extension of
$\mathbb Q_p$ with
$p\ne 2$. Every square
matrix $A$ over\/
$\mathbb F$ is
*congruent to a direct
sum that is uniquely determined up to
permutation of summands and consists of
any number of summands of the following
types:
\begin{itemize}
 \item [{\rm(i)}]
$J_n(0)$;

 \item [{\rm(ii)}]
$[\Phi\diag I_n]$, in
which $\Phi$ is an
$n\times n$
nonsingular Frobenius
block over\/ $\mathbb
F$ such that
$p_{\Phi}(x) \ne
p_{\Phi}^{\vee}(x)$ $($see \eqref{ksu}$)$
and $\Phi$ is
determined up to
replacement by the
Frobenius block $\Psi$
with $\chi_{\Psi}(x)
=\chi^{\vee}_{\Phi}(x)$;

 \item [{\rm(iii)}]
for each nonsingular
Frobenius block $\Phi$
over\/ $\mathbb F$
such that $p_{\Phi}(x)
= p_{\Phi}^{\vee}(x)$:
\begin{itemize}
  \item[$\bullet$]
$\sqrt[\displaystyle
*]{\Phi}$,
  \item[$\bullet$]
at most one summand of
the form
\[
  \begin{cases}
\sqrt[\displaystyle
*]{\Phi} \tilde{ \pi}
(\Phi) & \text{if
$\mathbb K=\mathbb
K_{\circ}
(\sqrt { u})$,}\\
\sqrt[\displaystyle
*]{\Phi}\tilde{
u}(\Phi) & \text{if
$\mathbb K=\mathbb
K_{\circ}(\sqrt {
\pi})$ or $\mathbb
K=\mathbb
K_{\circ}(\sqrt {{ u}{
\pi}})$},
  \end{cases}
\]
in which $\mathbb K$ is
the following field
with involution:
\[
\mathbb K:= \mathbb
F(\kappa)=\mathbb
F[x]/p_{\Phi}(x)\mathbb
F[x],\qquad
f(\kappa)^{\circ} :=
\bar f(\kappa^{-1})\]
$($see \eqref{iut}, \eqref{alft}
and \eqref{alfta}$)$,
$\mathbb K_{\circ}$ is
its fixed field, ${
\pi}$ is a prime
element of $\mathbb
K_{\circ}$, ${ u}\in
{\cal O}(\mathbb
K_{\circ})^{\times}
\setminus \mathbb
K_{\circ}^{\times 2}$
is a unit that is not
a square, $\tilde{
\pi}(x),\tilde{
u}(x)\in \mathbb
F[x,x^{-1}]$ are the
functions of the form
\eqref{ser13} such
that $\tilde{
\pi}(\kappa)={ \pi}$
and $\tilde{
u}(\kappa)={ u}$.
\end{itemize}
\end{itemize}
\end{theorem}

\begin{proof}
Let a
field $\mathbb F$ with
nonidentity involution
be a
finite extension of
$\mathbb Q_p$ with
$p\ne 2$.
By Theorem
\ref{Theorem 5}(a),
every square matrix
$A$ over $\mathbb F$
is *congruent to a
direct sum of matrices
of the form
\[
\text{(a)}\
J_n(0),\quad
\text{(b)}\ [\Phi\diag
I_n]\text{ if
$\sqrt[\displaystyle
*]{\Phi}$ does not
exist},\quad\text{(c)}\
\sqrt[\displaystyle
*]{\Phi}q(\Phi).
\]

Consider each of these
summands.

\emph{Summands }(a).
Theorem \ref{Theorem
5}(b) ensures that the
summands of the form
$J_n(0)$ are uniquely
determined by $A$,
which gives the
summands (i) of the
theorem.

\emph{Summands }(b).
By Lemma
\ref{lsdy1}(a),
$\sqrt[\displaystyle
*]{\Phi}$ does not
exist if and only if $
p_{\Phi}(x) \ne
p_{\Phi}^{\vee}(x)$.
Theorem \ref{Theorem
5}(b) ensures that the
summands of the form
$[\Phi\diag I_n]$ are
uniquely determined by
$A$,  up to replacement
of $\Phi$ by $\Psi$
with $\chi_{\Psi}(x)
=\chi^{\vee}_{\Phi}(x)$.
This gives the
summands (ii).

\emph{Summands }(c).
Let $\Phi$ be a
nonsingular $n\times
n$ Frobenius block for
which
$\sqrt[\displaystyle
*]{\Phi}$ exists; this
means that $
p_{\Phi}(x) =
p_{\Phi}^{\vee}(x)$.
Consider the whole
group of summands of the form $\sqrt[\displaystyle
*]{\Phi}q(\Phi)$ with the same $\Phi$:
\begin{equation}\label{kuys}
\sqrt[\displaystyle
*]{\Phi}q_1(\Phi)
\oplus\dots\oplus
\sqrt[\displaystyle
*]{\Phi}q_s(\Phi).
\end{equation}
The involution
$f(\kappa)^{\circ} :=
f(\kappa^{-1})$ on the
field $\mathbb
F(\kappa) = \mathbb
F[x]/p_{\Phi}(x)\mathbb
F[x]$ (see
\eqref{alft} and
\eqref{alfta}) is
nonidentity since it
extends the
nonidentity involution
on $\mathbb F$. By
Lemma \ref{kurt} with
$\mathbb K:=\mathbb
F(\kappa)$, the
Hermitian form
\[q_1(\kappa)x_1^{\circ}y_1+\dots+
q_s(\kappa)x_s^{\circ}y_s\]
over $\mathbb
F(\kappa)$ is
equivalent to either
$x_1^{\circ}y_1+\dots+
x_s^{\circ}y_s$, or $tx_1^{\circ}y_1+ x_2^{\circ}y_2+\dots+
x_s^{\circ}y_s$, in
which $t$ is defined
in \eqref{fsr}. By
Theorem \ref{Theorem
5}(b), the matrix
\eqref{kuys} is
*congruent to
\[
\text{either }\  \sqrt[\displaystyle
*]{\Phi}
\oplus\dots\oplus
\sqrt[\displaystyle
*]{\Phi},\qquad\text{or
}\ \sqrt[\displaystyle
*]{\Phi}\tilde t(\Phi)
\oplus\sqrt[\displaystyle
*]{\Phi}
\oplus\dots\oplus
\sqrt[\displaystyle
*]{\Phi},
\]
where $\tilde t(x)\in
\mathbb F[x,x^{-1}]$
is the function of the
form \eqref{ser13}
such that
$\tilde{t}(\kappa)=t$.
This sum is uniquely
determined by $A$,
which gives the
summands (iii).
\end{proof}

\subsection{Canonical pairs of
symmetric or
skew-symmetric
matrices}

For each Frobenius
block $\Phi$, denote
by $\sqrt[T]{\Phi}$
and
$\Phi_{\varepsilon}$
($\varepsilon=\pm 1$)
fixed nonsingular
matrices satisfying,
respectively, the
conditions
\begin{align}\label{vrfmau1}
\sqrt[T]{\Phi}&=(\sqrt[T]{\Phi})^TA,
   \\ \label{vrfmau2a}
\Phi_{\varepsilon}
&=\Phi_{\varepsilon}^T,
\quad
\Phi_{\varepsilon}A=
\varepsilon
(\Phi_{\varepsilon}A)^T,
\end{align}
in which $A$ is
similar to $\Phi$.
Each of these matrices
may not exist for
some $\Phi$; existence
conditions and
explicit forms of
these matrices were
established in
\cite{ser_izv}.

\begin{theorem}\label{Them 4a}
Let a field $\mathbb F$ be a finite
extension of $\mathbb Q_p$ with $p\ne
2$. Each pair of symmetric matrices of
the same size over $\mathbb F$ is
congruent to a direct sum that is
uniquely determined up to permutation
of summands and consists of any number
of summands of the following types:
\begin{itemize}
  \item[\rm(i)]
$([F_n\diag
 F_n^T],\,
[G_n\diag G_n^T])$,
in which $F_n$ and $G_n$
are defined in
\eqref{ser17};

  \item[\rm(ii)]
for each nonsingular
Frobenius block
${\Phi}$ over $\mathbb
F$:
\begin{itemize}
  \item[$\bullet$]
$(\Phi_{1},
\Phi_{1}\Phi)$, in
which $\Phi_1$ is
defined in
\eqref{vr2},

  \item[$\bullet$]
at most one summand of
the form
\[
(\Phi_{1},
\Phi_{1}\Phi)
f_1(\Phi)\oplus\dots\oplus
(\Phi_{1},
\Phi_{1}\Phi)f_t(\Phi),
\]
in which
$(f_1(x),\dots,f_t(x))$
is a sequence of
polynomials over
$\mathbb F$ of degree
$<\deg(p_{\Phi}(x))$
such that the sequence
$(f_1(\omega),
\dots,f_t(\omega))$ of
elements of the field
\[\mathbb K:={\mathbb
F}(\omega)={\mathbb
F}[x]/p_{\Phi}(x){\mathbb
F}[x]\] is one of the
sequences
\begin{equation}\label{glrs}
(1),\ ( u),\ ( \pi),\
( u \pi),\ ( u, \pi),\
( u, u \pi),\ ( \pi,r
\pi), \ ( u, \pi, r
\pi),
\end{equation}
where
\begin{equation}\label{fepa}
r:=\begin{cases}
 u &
\text{if $p^m\equiv 1\mod 4$}, \\
 1& \text{if
$p^m\equiv 3\mod 4$},
  \end{cases}
\end{equation}
$p^m$ is the
number of elements of
the residue field
${\cal O}(\mathbb
K)/\mathfrak m$ of
$\mathbb K$, $ u\in
{\cal O}(\mathbb
K)^{\times}\setminus
\mathbb K^{\times 2}$
is a unit that is not
a square, and $ \pi$
is a prime element of
$\mathbb K$;
\end{itemize}

  \item[\rm(iii)]
for each
$n=1,2,\dots$:
\begin{itemize}
  \item[$\bullet$]
the pair of
$n$-by-$n$ matrices
\begin{equation}\label{ser20}
B_n:= \left(\!
 \begin{bmatrix}
0&&1&0\\
&\ddd&\ddd\\
1&0\\
0&&&0
\end{bmatrix},
\begin{bmatrix}
0&&&1\\
&&\ddd\\
&1\\1&&&0
\end{bmatrix}
 \!\right),
\end{equation}

\item[$\bullet$] at
most one summand of
the form
\begin{equation}\label{flvr}
B_nc_1
\oplus\dots\oplus
B_nc_t,
\end{equation}
in which
$(c_1,\dots,c_t)$ is
one of the sequences
\begin{equation}\label{grka}
(1),\ (u),\ (\pi),\ (u
\pi),\ (u,\pi),\ (u,u
\pi),\ (\pi,r \pi),\
(u,\pi, r \pi),
\end{equation}
where
\begin{equation}\label{fega}
r:= \begin{cases} u &
\text{if $p^m\equiv 1\mod 4$}, \\
1 & \text{if
$p^m\equiv 3\mod 4$},
  \end{cases}
\end{equation}
$p^m$ is the number of
elements of the
residue field ${\cal
O}(\mathbb
F)/\mathfrak m$ of
$\mathbb F$, $u\in
{\cal O}(\mathbb
F)^{\times}\setminus
\mathbb F^{\times 2}$
is a unit that is not
a square, and $\pi$ is
a prime element of
$\mathbb F$.
\end{itemize}
\end{itemize}
\end{theorem}

\begin{proof}
Let a
field $\mathbb F$ with the identity involution
be a
finite extension of
$\mathbb Q_p$,
$p\ne 2$. By Lemma
\ref{THEjM}, the
matrix $\Phi_1$ exists
for each nonsingular
Frobenius block $\Phi$
over $\mathbb F$. By
Theorem \ref{The4}(a),
each pair $(A,B)$ of
symmetric matrices of
the same size is
congruent to a direct
sum of pairs of the
form
\[
\text{(a)}\ ([F_n\diag
F_n^T],\, [G_n\diag
G_n^T]),\quad\text{(b)}\
A_{\Phi}^{f(x)}:=
(\Phi_1,
\Phi_1\Phi)f(\Phi),
\quad\text{(c)}\
B_n^a,
\]
in which $f(x)\in
\mathbb F[x]$ is a
nonzero polynomial of
degree
$<\deg(p_{\Phi}(x))$
and $0\ne a\in \mathbb F$.

Consider each of these
summands.

\emph{Summands }(a).
Theorem \ref{The4}(b)
ensures that the
summands of the form
(a) are uniquely
determined by $(A,B)$,
which gives the
summands (i) of the
theorem.

\emph{Summands }(b).
Consider the whole
group of summands of the form $A_{\Phi}^{g(x)}$ with the same
nonsingular Frobenius
block $\Phi$:
\begin{equation}\label{kuy3a1}
A_{\Phi}^{g_1(x)}
\oplus\dots\oplus
  A_{\Phi}^{g_s(x)}.
\end{equation}
By Lemma
\ref{kux}, the
quadratic form
\[q_1(\omega)x_1^2+\dots+
q_s(\omega)x_s^2\]
over ${ \mathbb
F}(\omega)={ \mathbb
F}[x]/p_{\Phi}(x){\mathbb
F}[x]$ is equivalent
to exactly one form
\eqref{ghts}, in which
$(a_1,\dots,a_t)$ is
one of the sequences
\eqref{gjrs}. Theorem
\ref{The4}(b) ensures
that \eqref{kuy3a1} is
congruent to a direct
sum of pairs of the
form (ii) and this sum
is uniquely determined
by $(A,B)$, which gives
the summands (ii).

\emph{Summands }(c).
For each $n$, consider
the whole group of
summands of the form $B_n^{a}$ with the same $n$:
\begin{equation}\label{kuy3b1}
B_n^{a_1}
\oplus\dots\oplus
 B_n^{a_s}.
\end{equation}
By Lemma \ref{kux},
the quadratic form
$a_1x_1^2+\dots+
a_sx_s^2$ over
${\mathbb F}$ is
equivalent to to
exactly one form
\eqref{ghts}, in which
$(c_1,\dots,c_t)$ is
one of the sequences
\eqref{gjrs}. Theorem
\ref{The4}(b) ensures
that \eqref{kuy3b1} is
congruent to a direct
sum of pairs of the
form (iii) and this
sum is uniquely
determined by $(A,B)$, which
gives the
summands (iii).
\end{proof}

\begin{theorem}\label{Theorem 4b}
Let a field $\mathbb F$ be a finite
extension of $\mathbb Q_p$ with $p\ne
2$. Each pair consisting of a symmetric
matrix and a skew-symmetric matrix of
the same size over $\mathbb F$ is
congruent to a direct sum   that is
uniquely determined up to permutation
of summands and consists of any number
of summands of the following types:
\begin{itemize}
  \item[\rm(i)]
$([F_n\diag
 F_n^T],\,
[G_n\diag - G_n^T])$,
in which $F_n$ and
$G_n$ are defined in
\eqref{ser17};

  \item[\rm(ii)]
$([I_n\diag I_n],\,
[\Phi\diag -
\Phi^T])$, in which
$\Phi$ is an $n\times
n$ Frobenius block
over $\mathbb F$ such
that
\begin{equation}\label{rad}
p_{\Phi}(x)\notin\mathbb
F[x^2], \qquad\Phi\ne
J_1(0),J_3(0),J_5(0),\dots
\end{equation}
$($see \eqref{fik}$)$,
and $\Phi$ is
determined up to
replacement by the
Frobenius block $\Psi$
with $\chi_{\Psi}(x)
=(-1)^{\det
\chi_{\Phi}}\chi_{\Phi}(-
x)$;

  \item[\rm(iii)]
for each Frobenius
block ${\Phi}$  over
$\mathbb F$ such that
$p_{\Phi}(x)\in\mathbb
F[x^2]$:
\begin{itemize}
  \item[$\bullet$]
$(\Phi_{-1},
\Phi_{-1}\Phi)$,

  \item[$\bullet$]
at most one summand of
the form
\[
  \begin{cases}
(\Phi_{-1},
\Phi_{-1}\Phi)
\tilde{\pi}(\Phi)
 &
\text{if $\mathbb K=
\mathbb
K_{\circ}(\sqrt
{ u})$,}\\
(\Phi_{-1},
\Phi_{-1}\Phi)\tilde{
u}(\Phi) & \text{if
$\mathbb K=\mathbb
K_{\circ}(\sqrt{
\pi})$ or $\mathbb
K=\mathbb
K_{\circ}(\sqrt {
u\pi})$},
  \end{cases}
\]
in which $\mathbb K$
is the following field
with involution:
\[
\mathbb K:={\mathbb
F}(\omega)={\mathbb
F}[x]/p_{\Phi}(x){\mathbb
F}[x],\qquad
f(\omega)^{\circ}=
f(-\omega),\] $\mathbb
K_{\circ}$ is its
fixed field, $\pi$ is
any prime element of
$\mathbb K_{\circ}$,
and $ u\in {\cal
O}(\mathbb
K_{\circ})^{\times}\setminus
\mathbb
K_{\circ}^{\times 2}$
is any unit that is
not a square; $\tilde{
\pi}(x)$ and $\tilde{
u}(x)$ are polynomials
over $\mathbb F$ of
degree
$<\deg(p_{\Phi}(x))$
such that
\[\tilde{\pi}(x)=
\tilde{\pi}(- x),\
\tilde{\pi}(\omega)=
\pi, \qquad \tilde{
u}(x)= \tilde{ u}(-
x),\ \tilde{
u}(\omega)= u;
\]
\end{itemize}

  \item[\rm(iv)]
$([J_n(0)\diag
J_n(0)^T], [I_n\diag
-I_n])$, in which $n$
is odd;

  \item[\rm(v)]
for each
$n=1,2,3,\dots$:
\begin{itemize}
  \item[$\bullet$]
the pair
of $n$-by-$n$
symmetric and
skew-symmetric
matrices defined as follows:
\begin{equation}\label{seEh}
C_n:=\left(\!
 \begin{bmatrix}
0&&1\\
&\ddd\\
1&&0
\end{bmatrix},
\begin{bmatrix}
0&&&&&1&0\\
&&&&\ddd&\ddd\\
&&&1&0&\\
&&-1&0&&\\
&\ddd&0&&&\\
-1&\ddd&&&&\\
0&&&&&&0
\end{bmatrix}
 \!\right)
\end{equation}
if $n$ is odd, and
\begin{equation}\label{seF}
C_n:=\left(\!
 \begin{bmatrix}
0&&1&0\\
&\ddd&\ddd\\
1&0\\
0&&&0
\end{bmatrix},
\begin{bmatrix}
0&&&&&1\\
&&&&\ddd&\\
&&&1&&\\
&&-1&&&\\
&\ddd&&&&\\
-1&&&&&0\\
\end{bmatrix}
 \!\right)
\end{equation}
if $n$ is even, and

\item[$\bullet$] at
most one summand of
the form
\begin{equation}\label{flG}
C_nc_1
\oplus\dots\oplus
C_nc_t,
\end{equation}
in which
$(c_1,\dots,c_t)$ is
one of the sequences
\begin{equation}\label{gGy}
(1),\ (u),\ (\pi),\ (u
\pi),\ (u,\pi),\ (u,u
\pi),\ (\pi,r\pi),\
(u,\pi, r\pi),
\end{equation}
where
\begin{equation}\label{feG}
r:=\begin{cases} u&
\text{if $p^m\equiv 1\mod 4$}, \\
1 & \text{if
$p^m\equiv 3\mod 4$},
  \end{cases}
\end{equation}
$p^m$ is the number of
elements of the
residue field ${\cal
O}(\mathbb
F)/\mathfrak m$ of
$\mathbb F$, $u\in
{\cal O}(\mathbb
F)^{\times}\setminus
\mathbb F^{\times 2}$
is a unit that is not
a square, and $\pi$ is
a prime element of
$\mathbb F$.
\end{itemize}

\end{itemize}
\end{theorem}

\begin{proof}
Let a field $\mathbb F$ with the
identity involution be a
finite extension of
$\mathbb Q_p$,
$p\ne 2$.
By Theorem
\ref{The4}(a) and
Remark \ref{hts}, each
pair $(A,B)$
consisting of a
symmetric matrix $A$
and a skew-symmetric
matrix $B$ of the same
size is congruent to a
direct sum of pairs of
the form
\begin{itemize}
  \item[(a)]
$([F_n\diag F_n^T],\,
[G_n\diag -G_n^T])$,

  \item[(b)] $([I_n\diag
 I_n],\,
[\Phi\diag - \Phi^T])$
if $\Phi_{-1}$ does
not exist,

  \item[(c)]
$A_{\Phi}^{f(x)}:=
(\Phi_{-1},
\Phi_{-1}\Phi)f(\Phi)$,
in which $0\ne f(x)=
f(- x)\in \mathbb
F[x]$ and
$\deg(f(x))<\deg(p_{\Phi}(x))$,

 \item[(d)]
$([J_n(0)\diag
J_n(0)^T], [I_n\diag
-I_n])$, in which $n$
is odd,

 \item[(e)]
$ C_n^a$ (defined in
\eqref{serk1}), in
which $n$
is even and $ 0\ne a\in \mathbb
F$.
\end{itemize}

Consider each of these
summands.

\emph{Summands }(a).
Theorem \ref{The4}(b)
ensures that the
summands of the form
(a) are uniquely
determined by $(A,B)$,
which gives the
summands (i) of the
theorem.

\emph{Summands }(b).
By Lemma \ref{THEjM},
$\Phi_{-1}$ does not
exist if and only if
\eqref{rad} is
satisfied. Theorem
\ref{The4}(b) ensures
that the summands of
the form (b) are
uniquely determined by
$(A,B)$, up to
replacement of $\Phi$
by $\Psi$ with
$\chi_{\Psi}(x)
=(-1)^{\det
\chi_{\Phi}}\chi_{\Phi}(-
x)$, which gives the
summands (ii).

\emph{Summands }(c).
By Lemma \ref{THEjM},
$\Phi_{-1}$ exists if
and only if
\eqref{rad} is not
satisfied; that is,
\begin{equation}\label{ytw}
p_{\Phi}(x)\in\mathbb
F[x^2]\quad\text{or}\quad
\Phi=
J_1(0),J_3(0),J_5(0),\dots
\end{equation}
Consider the whole
group of summands of the form $A_{\Phi}^{f(x)}$
with the same
nonsingular Frobenius
block $\Phi$:
\begin{equation}\label{kuy3s1}
A_{\Phi}^{f_1(x)}
\oplus\dots\oplus
  A_{\Phi}^{f_s(x)}.
\end{equation}

Let first
$p_{\Phi}(x)\in\mathbb
F[x^2]$. Then the
involution
$f(\omega)^{\circ}=
f(-\omega)$ on ${
\mathbb F}(\omega)={
\mathbb
F}[x]/p_{\Phi}(x){\mathbb
F}[x]$  is nonidentity
(since $
\omega^{\circ}=-\omega\ne
\omega$). By Lemma
\ref{kurt}, the
Hermitian form
\[f_1(\omega)x_1^{\circ}y_1+\dots+
f_s(\omega)x_s^{\circ}y_s\]
over ${ \mathbb
F}(\omega)$ is
equivalent to either
$x_1^{\circ}y_1+\dots+
x_s^{\circ}y_s$, or $tx_1^{\circ}y_1+x_2^{\circ}y_2+\dots+
x_s^{\circ}y_s$, in
which $t$ is defined
in \eqref{fsr}.
Theorem \ref{Theorem
5}(b) ensures that
\eqref{kuy3s1} is
congruent to
\[
\text{either }\  A_{\Phi}^1
\oplus\dots\oplus
A_{\Phi}^1,\qquad\text{or
}\ A_{\Phi}^1\tilde
t(\Phi) \oplus
A_{\Phi}^1
\oplus\dots\oplus
A_{\Phi}^1,
\]
where $\tilde t(x)\in
\mathbb F[x,x^{-1}]$
is the function of the
form \eqref{ser13}
such that
$\tilde{t}(\kappa)=t$.
This sum is uniquely
determined by $(A,B)$,
which gives the
summands (iii).

Let now $\Phi= J_n(0)$
with $n=2m+1$ and
$m=1,2,\dots$ The
equalities \eqref{vr2}
hold for the $n\times
n$ matrices $\Phi'$ and $
\Phi_{-1}'$ defined in
\eqref{jtw} instead of
$\Phi$ and $
\Phi_{-1}$. Since
$\Phi$ and $\Phi'$
are similar, by
Theorem \ref{The4}(c)
we can take
$C_n^{f(x)}:=(\Phi'_{-1},
\Phi'_{-1}\Phi')f(\Phi')$
instead of
$A_{\Phi}^{f(x)}$ and
\begin{equation}\label{kuj1}
C_n^{f_1(x)}
\oplus\dots\oplus
  C_n^{f_s(x)}
\end{equation}
instead of
\eqref{kuy3s}.

Since $p_{\Phi}(x)=x$, the field $\mathbb
F(\omega) = \mathbb
F[x]/p_{\Phi}(x)\mathbb
F[x]$ is $\mathbb F$
with the identity
involution and all
polynomials $f_i(x)$
in \eqref{kuj} are
some scalars
$a_i\in\mathbb F$. By
Lemma \ref{kux}, the
quadratic form
$a_1x_1^2+\dots+
a_sx_s^2$ over
${\mathbb F}$ is
equivalent to exactly
one form \eqref{ghts},
in which
$(c_1,\dots,c_t)$ is
one of the sequences
\eqref{gjrs}. Theorem
\ref{The4}(b) ensures
that \eqref{kuj1} is
congruent to a direct
sum of pairs of the
form (iii), and this
sum is uniquely
determined by $(A,B)$.
This gives the
summands (v) with odd
$n$.

\emph{Summands }(d).
Theorem \ref{The4}(b)
ensures that the
summands of the form
(d) are uniquely
determined by $(A,B)$,
which gives the
summands (iv).

\emph{Summands }(e).
Consider the whole
group of summands of
the form $C_n^a$ with the same $n$:
\begin{equation}\label{kur1}
C_n^{a_1}
\oplus\dots\oplus
 C_n^{a_s}.
\end{equation}
By Lemma \ref{kux},
the quadratic form
$a_1x_1^2+\dots+
a_sx_s^2$ over
${\mathbb F}$ is
equivalent to exactly
one form \eqref{ghts},
in which
$(c_1,\dots,c_t)$ is
one of the sequences
\eqref{gjrs}. Theorem
\ref{The4}(b) ensures
that \eqref{kur1} is
congruent to a direct
sum of pairs of the
form (iii), and this
sum is uniquely
determined by $(A,B)$.
This gives the
summands (v) with even
$n$.
\end{proof}

\subsection{Canonical pairs of
Hermitian matrices}

\begin{theorem}\label{Thew}
Let a field $\mathbb
F$ with nonidentity
involution be a finite
extension of $\mathbb
Q_p$, $p\ne 2$.
Let $\mathbb
F_{\circ}$ be the
fixed field of
\/$\mathbb F$. Each pair
of Hermitian matrices
of the same size over
$\mathbb F$ is
*congruent to a direct
sum that is uniquely determined up to
permutation of summands and consists of
any number of summands of the following
types:
\begin{itemize}
  \item[\rm(i)]
$([F_n\diag
 F_n^*],\,
[G_n\diag G_n^*])$, in which
$F_n$ and $G_n$
are defined in
\eqref{ser17};

 \item[\rm(ii)]
$([I_n\diag I_n],\,
[\Phi\diag \Phi^*])$,
in which $\Phi$ is an
$n\times n$ Frobenius
block such that
$p_{\Phi}(x)\notin\mathbb
F_{\circ}[x]$, and
$\Phi$ is determined
up to replacement by
the Frobenius block
$\Psi$ with
$\chi_{\Psi}(x)
=\bar\chi_{\Phi}(x)$ $($see \eqref{iut}$)$;

  \item[\rm(iii)]
for each Frobenius
block ${\Phi}$ over
$\mathbb F$ such that
$p_{\Phi}(x)\in\mathbb
F_{\circ}[x]$:
\begin{itemize}
  \item[$\bullet$]
$(\Phi_{1},
\Phi_{1}\Phi)$, in
which $\Phi_1$ is
defined in Lemma
\ref{THEjM},

  \item[$\bullet$]
at most one summand of
the form
\[
  \begin{cases}
(\Phi_{1},
\Phi_{1}\Phi)
f_{\pi}(\Phi)
 &
\text{if $\mathbb K=
\mathbb
K_{\circ}(\sqrt
{ u})$,}\\
(\Phi_{1},
\Phi_{1}\Phi)f_{
u}(\Phi) & \text{if
$\mathbb K=\mathbb
K_{\circ}(\sqrt{
\pi})$ or $\mathbb
K=\mathbb
K_{\circ}(\sqrt {
u\pi})$},
  \end{cases}
\]
in which $\mathbb K$
is the following field
with involution:
\[
\mathbb K:={\mathbb
F}(\omega)={\mathbb
F}[x]/p_{\Phi}(x){\mathbb
F}[x],\qquad
f(\omega)^{\circ}=
\bar f(\omega),\]
$\mathbb K_{\circ}$ is
its fixed field, $\pi$
is any prime element
of $\mathbb
K_{\circ}$, and $ u\in
{\cal O}(\mathbb
K_{\circ})^{\times}\setminus
\mathbb
K_{\circ}^{\times 2}$
is any unit that is
not a square, $\tilde{
\pi}(x)$ and $\tilde{
u}(x)$ are polynomials
over $\mathbb
F_{\circ}$ of degree
$<\deg(p_{\Phi}(x))$
such that
\[
\tilde{\pi}(\omega)=
\pi, \qquad \tilde{
u}(\omega)=
 u;
\]
\end{itemize}

  \item[\rm(iv)]
for each
$n=1,2,\dots$:
\begin{itemize}
  \item[$\bullet$]
the pair of
$n$-by-$n$
 matrices
\begin{equation*}\label{jer20}
B_n:= \left(\!
 \begin{bmatrix}
0&&1&0\\
&\ddd&\ddd\\
1&0\\
0&&&0
\end{bmatrix},
\begin{bmatrix}
0&&&1\\
&&\ddd\\
&1\\1&&&0
\end{bmatrix}
 \!\right),
\end{equation*}

\item[$\bullet$] at
most one summand of
the form
\[
  \begin{cases}
B_n \pi& \text{if
$\mathbb F=\mathbb
F_{\circ}
(\sqrt {u})$,}\\
B_nu & \text{if
$\mathbb F=\mathbb
F_{\circ}(\sqrt
{\pi})$ or $\mathbb
F=\mathbb
F_{\circ}(\sqrt
{{u}{\pi}})$},
  \end{cases}
\]
in which ${\pi}$ is a
prime element of
$\mathbb F_{\circ}$
and ${u}\in {\cal
O}(\mathbb
F_{\circ})^{\times}
\setminus \mathbb
F_{\circ}^{\times 2}$
is a unit that is not
a square.
\end{itemize}
\end{itemize}
\end{theorem}

\begin{proof}
Let a field $\mathbb
F$ with nonidentity
involution be a finite
extension of $\mathbb
Q_p$, $p\ne 2$. By
Theorem \ref{The4}(a),
each pair $(A,B)$ of
Hermitian matrices of
the same size over
$\mathbb F$ is
*congruent to a direct
sum of pairs of the
form
\begin{itemize}
  \item[\rm(a)]
$([F_n\diag F_n^*],\,
[G_n\diag G_n^*])$,

  \item[\rm(b)]
$([I_n\diag I_n],\,
[\Phi\diag \Phi^*])$
if $\Phi_1$ does not
exist,

  \item[\rm(c)]
$A_{\Phi}^{f(x)}:=
(\Phi_{\delta},
\Phi_1\Phi)f(\Phi)$,
in which $0\ne
f(x)=\bar f(x)\in
\mathbb F[x]$ and
$\deg(f(x))<\deg(p_{\Phi}(x))$.

  \item[\rm(d)] $B_n^a$ (defined in
\eqref{ser120}), in
which $0\ne a=\bar
a\in \mathbb F$.
\end{itemize}

Consider each of these
summands.

\emph{Summands }(a).
Theorem \ref{The4}(b)
ensures that the
summands of the form
(a) are uniquely
determined by $(A,B)$,
which gives the
summands (i) of the
theorem.

\emph{Summands }(b).
By Lemma
\ref{THEjM}(a),
${\Phi}_1$ does not
exist if and only if $
p_{\Phi}(x) \ne \bar
p_{\Phi} (x)$; that
is, $
p_{\Phi}(x)\notin\mathbb
F_{\circ}[x]$. Theorem
\ref{The4}(b) ensures
that the summands of
the form (b) are
uniquely determined,
up to replacement of
$\Phi$ by $\Psi$ with
$\chi_{\Psi}(x)
=\bar\chi_{\Phi}(x)$.
This gives the
summands (ii).

\emph{Summands }(c).
Consider the whole
group of summands of the form $A_{\Phi}^{f(x)}$ with the same
nonsingular Frobenius
block $\Phi$:
\begin{equation}\label{ktj}
A_{\Phi}^{f_1(x)}
\oplus\dots\oplus
  A_{\Phi}^{f_s(x)}.
\end{equation}
 By Lemma
\ref{kurt}, the
Hermitian form
\[f_1(\omega)x_1^{\circ}y_1+\dots+
f_s(\omega)x_s^{\circ}y_s\]
over ${ \mathbb
F}(\omega)={ \mathbb
F}[x]/p_{\Phi}(x){\mathbb
F}[x]$ with involution
$f(\omega)^{\circ}=
\bar f(\omega)$ is
equivalent to either
$x_1^{\circ}y_1+\dots+
x_s^{\circ}y_s$, or
$tx_1^{\circ}y_1+x_2^{\circ}y_2+\dots+
x_s^{\circ}y_s$, in
which $t$ is defined
in \eqref{fsr}.
Theorem \ref{Theorem
5}(b) ensures that
\eqref{kuy3s1} is
*congruent to
\[
\text{either }\  A_{\Phi}^1
\oplus\dots\oplus
A_{\Phi}^1,\qquad\text{or
}\ A_{\Phi}^1\tilde
t(\Phi) \oplus
A_{\Phi}^1
\oplus\dots\oplus
A_{\Phi}^1,
\]
in which $\tilde t(x)\in
\mathbb F[x,x^{-1}]$
is the function of the
form \eqref{ser13}
such that
$\tilde{t}(\kappa)=t$.
This sum is uniquely
determined by $(A,B)$,
which gives the
summands (iii).

\emph{Summands }(d).
Consider the whole
group of summands of
the form $B_n^a$ with the same $n$:
\begin{equation}\label{kuyc}
B_n^{a_1}
\oplus\dots\oplus
 B_n^{a_s}.
\end{equation}
 By Lemma
\ref{kurt}, the
Hermitian form
\[a_1x_1^{\circ}y_1+\dots+
a_sx_s^{\circ}y_s\]
over ${ \mathbb F}$ is
equivalent either
$x_1^{\circ}y_1+\dots+
x_s^{\circ}y_s$, or
$tx_1^{\circ}y_1+x_2^{\circ}y_2+\dots+
x_s^{\circ}y_s$, in
which $t$ is defined
in \eqref{fsr}.
Theorem \ref{Theorem
5}(b) ensures that
\eqref{kuyc} is
*congruent to
\[
\text{either }\  B_n \oplus\dots\oplus
B_n,\qquad\text{or
}\ B_n\tilde t(\Phi)
\oplus B_n
\oplus\dots\oplus B_n,
\]
in which $\tilde t(x)\in
\mathbb F[x,x^{-1}]$
is the function of the
form \eqref{ser13}
such that
$\tilde{t}(\kappa)=t$.
This sum is uniquely
determined by $(A,B)$,
which gives the
summands (iv).
\end{proof}

\section{Appendix: Quadratic forms over finite extensions of $p$-adic fields}
\label{pad}

In this section, we recall some known
results on quadratic forms over finite
extensions of $p$-adic fields that are
used in the paper.

Let $\mathbb F$ be a field and let
$\nu$ be an \emph{exponential
variation} on $\mathbb F$; that is, a
map $\nu : \mathbb F\to \mathbb
R\cup\{+ \infty\}$ with the properties
\begin{gather}\label{lkk}
\nu (x)=+\infty\quad \Longleftrightarrow\quad x=0,\\
\min \{\nu (x), \nu (y)\}\le \nu (x+y),\\
\nu (x)+\nu (y)=\nu
(xy)
\end{gather}
for all $x,y\in\mathbb
F$.

For example, the field
$\mathbb Q_p$ of
$p$-adic numbers
possesses an
exponential variation
that is defined on
each nonzero $p$-adic
number as follows:
\begin{equation}\label{ksyk}
 v(a_zp^z+a_{z+1}p^{z+1}+\dots)=z,
\end{equation}
where
$a_i\in\{0,1,\dots,p-1\}$,
$a_z\ne 0$, and $z\in
\mathbb Z$.

In this section
$\mathbb F$ denotes a
finite extension of
$\mathbb Q_p$, $p\ne
2$. In this case the
exponential variation
\eqref{ksy} can be
extended to an
exponential variation
of $\mathbb F$. This
variation is unique
and is given by the
formula:
\begin{equation}\label{kiy}
\nu (a)=\frac 1
nv(N(a))\qquad
\text{for all
}a\in\mathbb F,
\end{equation}
in which $n:=(\mathbb
F:\mathbb Q_p)=\dim
_{\mathbb Q_p}\mathbb
F$ is the degree of
$\mathbb F$ over
$\mathbb Q_p$ and
$N(a)$ is the
\emph{norm} of $a$ in
$\mathbb F$ over
$\mathbb Q_p$; that
is, the determinant of
the linear mapping
$x\mapsto xa$ on
$\mathbb F$ as a
vector space over
$\mathbb Q_p$. If
$x^m+\alpha
_1x^{m-1}+\dots+\alpha
_m$ is the minimum
polynomial of
$a\in\mathbb F$ over
$\mathbb Q_p$ then the
variation \eqref{kiy}
can be also given by
the formula:
\begin{equation}\label{kiys}
\nu (a)=\frac 1 m
v(\alpha _m)\qquad
\text{for all
}a\in\mathbb F.
\end{equation}

Note that there exists
a natural number $e$
such that
$e\nu(\mathbb
F^{\times})=\mathbb
Z$. The ring
\begin{equation}
{\cal
O}:=\{x\in\mathbb
F\,|\,\nu (x)\ge 0\}
\end{equation}
is called the
\emph{ring of
integers} (with
respect to $\nu$);
\begin{equation}
\mathfrak
m:=\{x\in\mathbb
F\,|\,\nu (x)> 0\}=\pi
R
\end{equation}
is the unique maximal
ideal of $R$ and its
generator $\pi$ is
called a \emph{prime
element} (it is any
element of $\mathbb F$
with the smallest
positive $\nu(\pi)$;
that is,
$\nu(\pi)=1/e$); the
factor ring
\begin{equation}
{\cal O}/\mathfrak m
\end{equation}
is a field, which is
called the
\emph{residue field};
and the set
\begin{equation}
{\cal
O}^{\times}:=\{x\in\mathbb
F\,|\,\nu (x)= 0\}
\end{equation}
is the group of all
invertible elements of
${\cal O}$ (which are
called \emph{units}).

The residue field
${\cal O}/\mathfrak m$
is an extension of the
residue field $\mathbb
F_p$ of $\mathbb Q_p$
and
\begin{equation}\label{jyf}
e ({\cal O}/\mathfrak
m:\mathbb F_p)=
n=(\mathbb F:\mathbb
Q_p)
\end{equation}

By \cite[Section VI,
Theorem 2.2]{lam} or
\cite[Ch. 6, Facts
4.1]{sch} $\mathbb
F^{\times}/\mathbb
F^{\times 2}$ consists
of 4 cosets,
represented by $1,\
u,\ \pi,\ u\pi$, where
$u\in {\cal
O}^{\times}$ is a unit
with $u\notin \mathbb
F^{\times 2}$ (or,
which is equivalent,
with $u+\mathfrak
m\notin ({\cal
O}/\mathfrak
m)^{\times 2}$; see
\cite[Example
3.11]{gam}, recall
that $p\ne 2$).

The \emph{Hilbert
symbol} is defined for
$a,b\in \mathbb
F^{\times}$ by
\begin{equation}\label{mhg}
(a,b)_{\mathbb F}:=
  \begin{cases}
    1 & \text{if $ax^2+by^2$ represents $1$}, \\
    -1 & \text{otherwise}.
  \end{cases}
\end{equation}
The \emph{Hasse
invariant} of a form
$q\sim
a_1x_1^2+a_2x_2^2+\dots+a_rx_r^2$
with $a_1,\dots,a_r\in
\mathbb F^{\times}$ is
\begin{equation}\label{loy}
c(q):=\prod_{i<
j}(a_i,a_j)_{\mathbb
F}
\end{equation}
(see \cite[Ch. VIII,
p. 210]{mil}).

By \cite[Ch. VIII, Theorem 4.10]{mil},
two quadratic forms over $\mathbb F$
are equivalent if and only if they have
the same rank $n$, the same
discriminant $d$ (in $\mathbb
F^{\times}/\mathbb F^{\times 2}$), and
the same Hasse invariant. By \cite[Ch.
VIII, Proposition 4.11]{mil}, if $q$ is
a quadratic form of rank $r$, then
\begin{itemize}
  \item
If $r=1$ then
$c(q)=1$.
  \item
If $r=2$ and $c(q)=-1$
then $d(q)\ne -1$ (mod
squares).
\end{itemize}
Apart from these
constraints, every
triple $r\ge 1$,
$d\in\{1,\ u,\ \pi,\
u\pi\}$ (mod squares),
$c=\pm 1$ occurs as
the set of invariants
of a quadratic form
over $\mathbb F$.

\begin{theorem}
Let $\mathbb F$ be a
finite extension of
$\mathbb Q_p$ with
$p\ne 2$. Let its
residue field ${\cal
O}/\mathfrak m$
consist of $p^m$
elements. Let $u\in
{\cal
O}^{\times}\setminus
\mathbb F^{\times 2}$
be a unit that is not
a square, and $\pi$ be
a prime element. Then
each quadratic form of
rank $r\ge 1$ over
$\mathbb F$ is
equivalent to exactly
one form
\begin{equation}\label{gts}
a_1x_1^2+a_2x_2^2+\dots+a_tx_t^2
+x_{t+1}^2+\dots+x_r^2,
\end{equation}
in which
$(a_1,\dots,a_t)$ is
one of the sequences:
\begin{equation}\label{grs}
(1),\ (u),\ (\pi),\
(u\pi),\ (u,\pi),\
(u,u\pi),
\end{equation}
\begin{equation}\label{fea}
\begin{cases}
    (\pi,u\pi),\ (u,\pi,u\pi) & \text{if $p^m\equiv 1\mod 4$}, \\
    (\pi,\pi),\ (u,\pi,\pi) & \text{if $p^m\equiv 3\mod 4$}.
  \end{cases}
\end{equation}
\end{theorem}

\begin{proof} Let us show that the forms \eqref{gts} give all possible invariant triples $(r,d,c)$.

The forms \eqref{gts}
with $t=1$ and
$a_1\in\{1,\ u,\ \pi,\
u\pi\}$ give all
possible triples
$(r,d,c)$ with $c=1$;
in particular, all
possible triples with
$r=1$.

The remaining forms
\eqref{gts} have the
Hasse invariant $c=-1$
since:

$\bullet$
$(u,\pi)_{\mathbb
F}=-1$ by
\cite[p.\,53, Case
2]{gam}.

$\bullet$
$(u,u\pi)_{\mathbb
F}=(u,u)_{\mathbb
F}(u,\pi)_{\mathbb
F}=-1$ since
$(u,u)_{\mathbb F}=1$
by \cite[p.\,53, Case
1]{gam}.

$\bullet$
$(\pi,\pi)_{\mathbb
F}=(-1)^{(q-1)/2}$ by
\cite[p.\,53, Case
3]{gam}. Thus,
$(\pi,\pi)_{\mathbb
F}=-1$  if $p^m\equiv
3\mod 4$ and
$(\pi,u\pi)_{\mathbb
F}=(\pi,u)_{\mathbb
F}(\pi,\pi)_{\mathbb
F}=-1$  if $p^m\equiv
1\mod 4$.

$\bullet$ If
$p^m\equiv 3\mod 4$
then the Hasse
invariant of the form
with triple
$(u,\pi,\pi)$ is
$(u,\pi)_{\mathbb
F}(u,\pi)_{\mathbb
F}(\pi,\pi)_{\mathbb
F}=(\pi,\pi)_{\mathbb
F}=-1$. If $p^m\equiv
1\mod 4$ then the
Hasse invariant of the
form with triple
$(u,\pi,u\pi)$ is
$(u,\pi)_{\mathbb
F}(u,u\pi)_{\mathbb
F}(\pi,u\pi)_{\mathbb
F}=(u,u)_{\mathbb
F}(u,\pi)_{\mathbb
F}^3(\pi,\pi)_{\mathbb
F}=(u,\pi)_{\mathbb
F}(\pi,\pi)_{\mathbb
F}=-1$.

In particular, we have
3 invariant triples
with $r=2$,  $c=-1$,
and
\[
d\in\begin{cases}
    \{u\pi,\ \pi,\ u\} \text{ (mod squares)} & \text{if $p^m\equiv 1\mod 4$}, \\
    \{u\pi,\ \pi,\ 1\}\text{ (mod squares)}  & \text{if $p^m\equiv 3\mod 4$}.
  \end{cases}
\]
But if $r=2$ and
$c=-1$ then $d$ can
have only 3 values
(mod squares) since
$d\ne -1$ (mod
squares); thus, we
have all possible
invariant triples with
$r=2$.

We have all possible
invariant triples with
$r\ge 3$ and $c=-1$,
since then $d\in\{1,\
u,\ \pi,\ u\pi\}$ (mod
squares).
\end{proof}

\subsection{Irreducible polynomials over $\mathbb Q_p$}

Let $f(x)\in \mathbb
Z_p[x]$ be a monic
polynomial whose
reduction modulo $p$
is irreducible in
$\mathbb F_p[x]$. Then
$f(x)$ is irreducible
over $\mathbb Q_p$.
\cite[Corollary
5.3.8]{Gou}

The Eisenstein
criterion. Suppose
that the polynomial
$f(x) = x^n +
a_1x^{n-1} + \dots +
a_n\in \mathbb Z_p[x]$
satisfies the
conditions $p|a_i$ for
all $i$ and $p^2\nmid
a_n$. Then $f(x)$ is
irreducible over
$\mathbb Q_p$.
\cite[Theorem
5.5]{bak}.

Let $n$ and $m$ be
coprime natural
numbers. Then the
polynomial $x^n-p^m$
is irreducible over
$\mathbb Q_p$.
\cite[Theorem
5.3]{bak}.

\subsection{Hermitian forms over local rings}

\begin{theorem}
Let $\mathbb F$ be a
finite extension of
$\mathbb Q_p$, $p\ne
2$, with a fixed
nonidentity
involution. Let
$\mathbb F_{\circ}$ be
the fixed field with
respect to this
involution. Let $u\in
{\cal O}(\mathbb
F_{\circ})^{\times}\setminus
\mathbb
F_{\circ}^{\times 2}$
be a unit that is not
a square, and $\pi$ be
a prime element of
$\mathbb F_{\circ}$.
Then each regular
(=with nonzero
determinant) Hermitian
form over $\mathbb F$
is equivalent to
either
\[\bar x_1 y_1+\dots + \bar x_n y_n,\]
or
\[
  \begin{cases}
\pi \bar x_1
y_1+\bar x_2y_2+\dots
+\bar x_{n}y_{n} &
\text{if $\mathbb
F=\mathbb F_{\circ}(\sqrt u)$,}\\
u\bar x_1y_1+\bar x_2y_2+\dots +\bar x_{n}y_{n} & \text{if
$\mathbb F=\mathbb
F_{\circ}(\sqrt \pi)$
or $\mathbb
F_{\circ}(\sqrt
{u\pi})$}.
  \end{cases}
\]
\end{theorem}

\begin{proof}
Since $(\mathbb F:
\mathbb F_{\circ})=2$,
we have $\mathbb
F=\mathbb
F_{\circ}(\alpha )$,
where $\alpha $ is a
root of
$f(x)=x^2+2ax+b\in\mathbb
F_{\circ}[x]$. Write
$\lambda :=\alpha +a$,
then $(\lambda
-a)^2+2a(\lambda
-a)+b=\lambda
^2-a^2+b=0$.

Therefore, we can take
$\alpha $ such that
$\alpha ^2=\beta
\in\mathbb F_{\circ}$.
Moreover, $(\alpha
a)^2=\beta a^2$ for
each $a\in\mathbb
F_{\circ}$. But
$\mathbb
F_{\circ}^{\times}/\mathbb
F_{\circ}^{\times 2}$
consists of 4 cosets,
represented by $1,\
u,\ \pi,\ u\pi$.
Hence, we can take
$\alpha $ such that
\[\alpha ^2=\beta
\in\{1,\ u,\ \pi,\
u\pi\},\] then
$\mathbb F$ is
$\mathbb
F_{\circ}(\sqrt u)$,
or $\mathbb
F_{\circ}(\sqrt \pi)$,
or $\mathbb
F_{\circ}(\sqrt
{u\pi})$. Since
$\bar\alpha ^2=\beta
$, the involution on
$\mathbb F$ is
$c+d\alpha \mapsto
a-d\alpha$,
$c,d\in\mathbb
F_{\circ}$. The
element
\[N(c+d\alpha)=(c+d\alpha)(c-d\alpha)=c^2-d^2\beta
\in\mathbb F_{\circ}\]
is the \emph{norm} of
$c+d\alpha$. The set
$N(\mathbb
F^{\times})$ of norms
of all nonzero
elements is a group.
By \cite[Ch. 6, Fact
4.3]{sch}, \emph{the
norm residue group
$\mathbb
F_{\circ}^{\times}/N(\mathbb
F^{\times})$ consists
of $2$ elements}.
\begin{itemize}
  \item
Let $\mathbb F=\mathbb
F_{\circ}(\sqrt u)$.
Then $\alpha ^2=u$,
$N(c+d\alpha)=
c^2-d^2u$. If $\pi\in
N(\mathbb F^{\times})$
then there is
$c+d\alpha$ such that
$N(c+d\alpha)=
c^2-d^2u=\pi$. Then
$c^2-d^2u=0\mod \pi$;
i.e. $u=(c/d)^2\mod
\pi$. A contradiction.
Therefore, the cosets
of  $\mathbb
F_{\circ}^{\times}/N(\mathbb
F^{\times})$ are
represented by $1,\
\pi$.

  \item
Let $\mathbb F=\mathbb
F_{\circ}(\sqrt \pi)$.
Then $\alpha ^2=\pi$,
$N(\alpha)= -\pi$. But
$-\pi$ is a prime
element too, so $1,\
u,\ -\pi,\ -u\pi$
represent 4 cosets of
$\mathbb
F_{\circ}^{\times}/\mathbb
F_{\circ}^{\times 2}$.
Thus, the cosets of
$\mathbb
F_{\circ}^{\times}/N(\mathbb
F^{\times})$ are
represented by $1,\
u$.

  \item
Let $\mathbb F=\mathbb
F_{\circ}(\sqrt
{u\pi})$. Then $\alpha
^2=u\pi$, $N(\alpha)=
-u\pi$. But $-u\pi$ is
a prime element too.
Thus, the cosets of
$\mathbb
F_{\circ}^{\times}/N(\mathbb
F^{\times})$ are
represented by $1,\
u$.

\end{itemize}

Let
$\phi(x,y)=\alpha_1\bar x_1
y_1+\dots+
\alpha_n\bar x_ny_n$
be a regular (all
$\alpha_i$ are
nonzero) Hermitian
form over $\mathbb F$.
Then the determinant
$\det(\phi):=a_1\dots
a_nN(\mathbb
F^{\times})\in \mathbb
F_{\circ}^{\times}/N(\mathbb
F^{\times})$ is an
invariant of
$\phi(x)$. By
\cite[Ch. 10, Example
1.6(ii)]{sch}, regular
Hermitian forms over
$\mathbb F$ are
classified by
dimension and
determinant.
\end{proof}

\end{document}